\documentclass[a4paper,12pt]{article}
\usepackage[utf8]{inputenc}
\usepackage{amssymb}
\usepackage{amsmath}
\usepackage{MnSymbol} 
\usepackage{mathrsfs} 
\usepackage{cancel}
\usepackage{ulem}
\usepackage{mathtools}
\usepackage{blkarray} 
\usepackage{tikz}
\usetikzlibrary{arrows}
\usepackage{tikz-cd}
\usepackage{fancyhdr}
\usepackage{authblk}
\usepackage{xcolor}
\usepackage{booktabs}
\usepackage{algorithm}
\usepackage{algorithmic}
\usepackage{listings} 

\usepackage{hyperref}

\usepackage{caption}
\usepackage{subcaption}
\usepackage{makecell}

\usepackage{amsthm}
\newtheorem{definition}{Definition}
\newtheorem{example}{Example}

\newtheorem{theorem}{Theorem}
\newtheorem{lemma}{Lemma}

\DeclareMathOperator{\Ima}{Im} 

\author[1]{Benjamin Jones \thanks{Corresponding to Benjamin Jones. Email: BenjaminDanielJones@gmail.com}}
\author[1,2,3]{Guo-Wei Wei}
\affil[1]{Department of Mathematics, Michigan State University, MI, 48824, USA}
\affil[2]{Department of Electrical and Computer Engineering, Michigan State University, MI 48824, USA}
\affil[3]{Department of Biochemistry and Molecular Biology, Michigan State University, MI 48824, USA}

\makeatletter
    \renewcommand*{\@fnsymbol}[1]{\ensuremath{\ifcase#1\or \dagger\or *\or *\or
   \mathsection\or \else\@ctrerr\fi}}
\makeatother
\date{}

\title{PETLS: PErsistent Topological Laplacian Software}
\date{March 2026}

\begin{document}
\maketitle
\paragraph{Abstract} Persistent topological Laplacians are operators that provide persistent Betti numbers and additional multiscale geometric information through the eigenvalues of the persistent topological Laplacian matrix. We introduce a framework and novel algorithm to aid in the computation of persistent topological Laplacians. We implement existing and new persistent Laplacian algorithms in an efficient and flexible C++ library with Python bindings, titled PETLS: PErsistent Topological Laplacian Software. As part of this library, we interface with several complexes commonly used in topological data analysis (TDA), such as simplicial, alpha, directed flag, Dowker, and cellular Sheaf. Because increased efficiency broadens the set of computationally feasible applications, we provide recommendations on how to use algorithms and complexes for data analysis in machine learning.


\newpage
    \tableofcontents
\newpage

\section{Introduction} 

Topological data analysis (TDA) has emerged as a leading approach to understand the shape of complex data. It has been used in many applications including predicting drug resistance \cite{chen2025drug}, fMRI analysis \cite{rieck2020uncovering}, multi-modal sensing \cite{schrader2023topological}, and higher-order signal processing \cite{schaub2021signal}. Topological deep learning (TDL), introduced in 2017 \cite{cang2017topologynet}, is an increasingly prominent paradigm for incorporating topological structure into deep learning frameworks, creating interpretable learning models, and processing complex data \cite{tdl_position,hajij2022topological,pun2022persistent}. 

The most celebrated tool of TDA is persistent homology (PH), an algebraic topology technique that provides quantitative and qualitative representations of the shape of data at various scales. However, this shape quantifier does not always contain enough information to draw meaningful conclusions. The need for more powerful geometric and topological descriptors of data at multiple scales motivated the creation of persistent topological Laplacians (PTLs) \cite{wang2020PSG}, often called persistent Laplacians (PLs) or persistent combinatorial Laplacians. PTLs bridge between spectral theory and algebraic topology. Specifically, a PTL can be thought of as a generalization of the graph Laplacian to higher-dimensional multi-scale settings in a way that also reproduces the topological invariants of persistent homology through its harmonic spectra. Furthermore, the eigenvectors of PTLs have no counterpart in persistent homology and can be applied independently \cite{liu2025manifold}. In the non-persistent setting, the combinatorial Laplacian is a generalization of graph Laplacians to simplicial complexes \cite{eckmann1944harmonische}. 

Persistent combinatorial Laplacians, persistent Laplacians, and related spectral theory structures have spread to several problem domains, from theoretical analysis \cite{liu2024algebraic, memoli_PL}, computational algorithms\cite{dong2024faster, wang2021hermes},  early prediction of COVID-19 variant dominance \cite{CHEN2022106262} to random walks on simplicial complexes \cite{schaub2020random} and persistent Rayleigh quotients for single-cell signal detection \cite{hoekzema2022multiscale}. Numerical experiments in more than 30 datasets demonstrated that persistent Laplacians offer more accurate predictions than persistent homology in protein engineering \cite{qiu2023persistent}. Combinatorial Laplacians are incorporated in the major topological deep learning library TopoX \cite{hajij2024topox}. HERMES is an earlier software for computing persistent Laplacians \cite{wang2021hermes}. This persistent spectral approach was generalized to the persistent Dirac operator formulation  \cite{ameneyro2024quantum}. Similar approaches in differential topology \cite{chen2021evolutionary} and geometric topology \cite{jones2024khovanov} have also been  developed for data on differential manifolds and curves embedded in 3 space, respectively.  A survey of these PTL developments can be found in \cite{wei_Survey}. More broad perspectives on TDA and TDL beyond persistent homology are reviewed in  \cite{suTDL}

In Figure \ref{fig:pipeline}, we show a framework of how persistent topological Laplacians can be created and used. Given a filtration of complexes, we produce a collection of PTL matrices, compute their eigenvalues and eigenvectors, and summarize those eigenvalues and eigenvectors in different filtration pairs $(a,b)$. Then these eigenvalue summaries or the eigenvectors themselves \cite{grande2024disentangling, liu2025manifold} are used for analysis, possibly as features for machine learning algorithms.

\begin{figure}[htb!]
    \centering
    \includegraphics[width=6in]{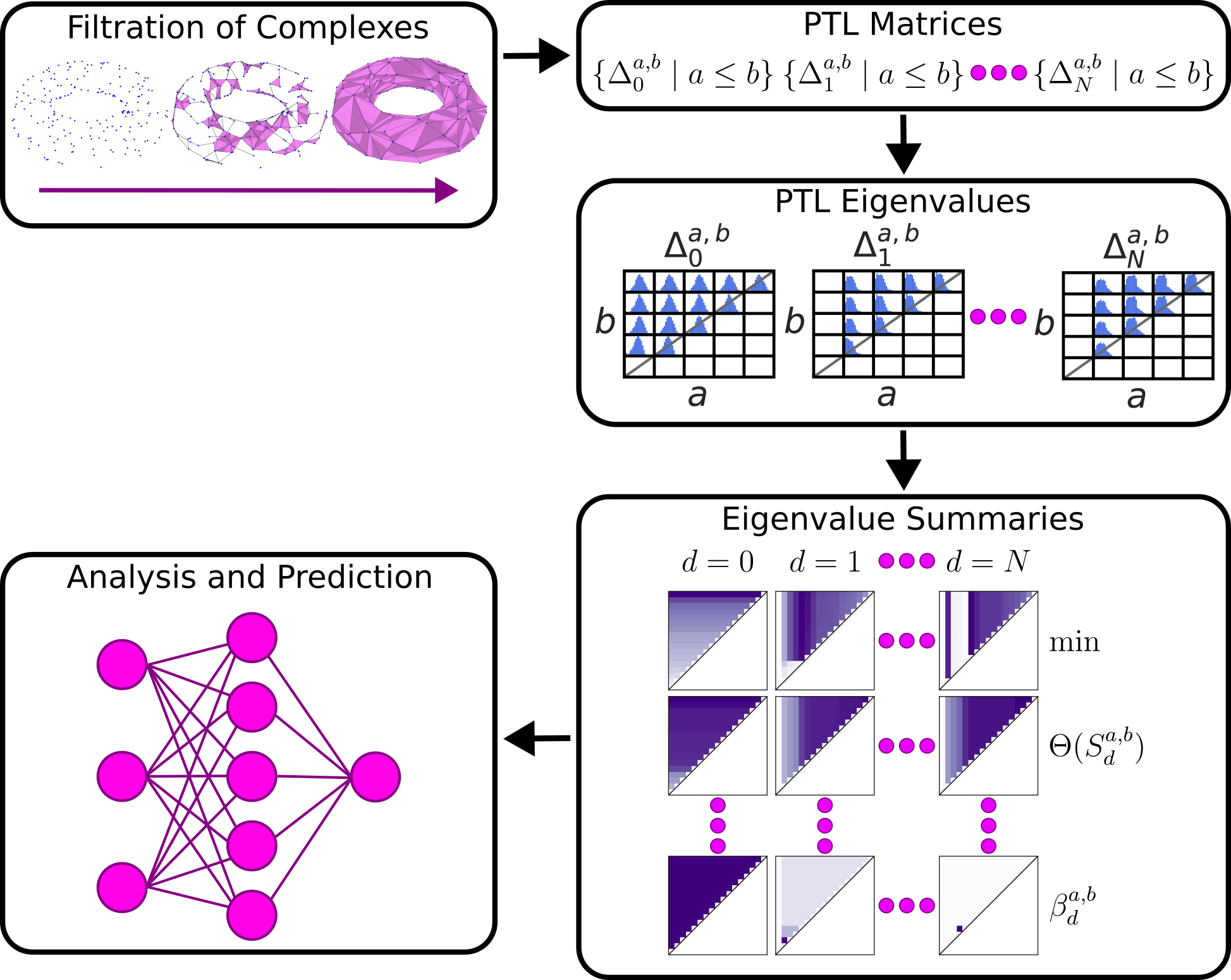}
    \caption{The process for creating and using persistent topological Laplacians (PTLs).  Pictured is an alpha filtration of a torus with maximum PTL dimension $N=2$, with summary function $\Theta$ as the maximum eigenvalue function.}
    \label{fig:pipeline}
\end{figure}

There are three main challenges limiting the broader impact of PTLs:
\begin{itemize}
    \item[]\textbf{Problem 1:} Scalability for large datasets.
    \item[]\textbf{Problem 2:} Ease of use for applied scientists and the TDA community
    \item[]\textbf{Problem 3:} Human interpretation of numerical results.
\end{itemize}

The objective of this work is to address these computational challenges. We make substantial progress on \textbf{Problem 1} and provide a robust solution for \textbf{Problem 2} through a flexible software library. These advances enable the computational scale and breadth of possible applications needed to address \textbf{Problem 3}, which we demonstrate by providing data-driven guidance on places for future investigation.

In Section \ref{sec:background}, we introduce the mathematics of the Persistent Topological Laplacian and highlight some key points of the literature. Definitions of some commonly used topological complexes and filtrations are given in Appendix \ref{appendix:complexes}. 

In Section \ref{sec:design}, we address \textbf{Problem 2}. We discuss the flexible design of the library, including the features of the C++ and Python versions. This discussion includes what types of simplicial and more general complexes are possible to study with the library.

In Section \ref{sec:computation}, we tackle \textbf{Problem 1}. We exploit structures of the PTL matrix for computational efficiency, which we implement in an extremely flexible C++ library. The improvements are substantial enough that the bottleneck of the pipeline shifts to computing the eigenvalues of the PTL matrix, rather than constructing the matrix itself. We address this through multiple avenues, including by introducing an algorithm that uses persistent homology to remove redundant information from the PTL. 

In Section \ref{sec:representations}, we show that our approaches to Problem 1 and Problem 2 yield the ability to compute persistent Laplacians of much larger complexes, which demonstrates where we should look to address \textbf{Problem 3}. 

Finally, in Section \ref{sec:discussion}, we discuss the impact of this work and how these advances guide the future directions of this extremely active area of research.

\section{Background}
\label{sec:background}

In Subsection \ref{subsec:combinatorial} we will build up the notions of simplicial complexes, homology, and combinatorial Laplacians. Then in Subsection \ref{subsec:persistence} we will introduce the notions of filtrations, persistent homology, and the persistent Laplacian. Then in Subsection \ref{subsec:computation} we will discuss the main computational approaches that have already been developed to calculate the persistent Laplacian.

\subsection{Combinatorial Laplacian}\label{subsec:combinatorial}

The fundamental constructions of many different persistent topological Laplacians are nearly the same, whether it is for simplicial complexes, directed simplicial complexes, path complexes, or hyperdigraph complexes. To keep the explanation clear for a wider audience by using standard terminology, we will work with simplicial complexes. It is important to note that when dealing with more sophisticated complexes that are not simplicial complexes, like the path complex, we only require a chain complex where every chain group has real coefficients and a finite basis. In practice, this means that most complexes used in topological data analysis can fit into the persistent topological Laplacian framework, and the software library is built to accommodate this level of generality.

We will now describe the standard concepts of an abstract simplicial complex, chains, a boundary map, and homology. The standard topology introduction for these topics is \cite{HatcherAT}, and a popular TDA-oriented introduction is \cite{otter_roadmap_2017}.

\begin{definition} An \textbf{abstract simplicial complex} on a finite vertex set $V$ is a collection $K$ of subsets of $V$ with the property that if $A\in K$ and $B\subset A$, then $B\in K$. A set in $K$ with $n+1$ elements is called an \textbf{$n$-simplex} (plural simplices), and the set of all $n$-simplices is denoted $K_n$. If $B\subset A\in K$, we call $B$ a \textbf{face} of $A$.
\end{definition}

The abstract simplicial complex is a convenient combinatorial perspective, but a geometric interpretation is often more useful for visualizing the shape of data. In Figure \ref{fig:simplicial}, we show a geometric simplicial complex as a union of geometric $n$-simplices, where $0$-simplices are vertices or points, $1$-simplices are edges connecting two points, $2$-simplices are triangles bordered by three edges, and $3$-simplices are tetrahedra bordered by four $2$-simplices.

\begin{figure}[htpb]
    \centering
    \includegraphics[width=4.325cm]{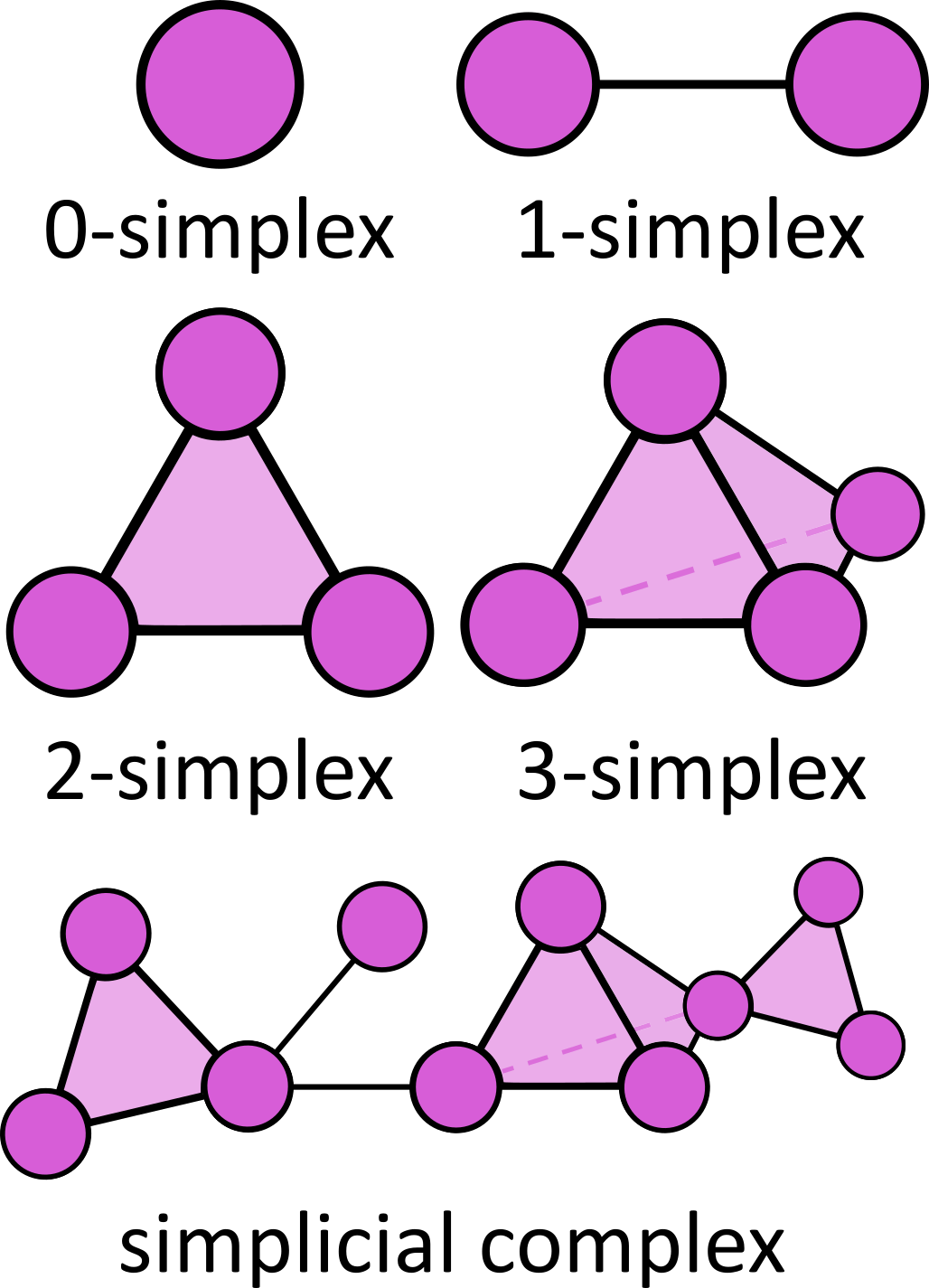}
    \caption{Simplices and a simplicial complex}
    \label{fig:simplicial}
\end{figure}

Next, we define chain groups and boundary map. We must fix a coefficient system $\mathbb{K}$, which is often $\mathbb{Z},\mathbb{Z}_2,\mathbb{Q},$ or $\mathbb{R}$. Most TDA settings assume $\mathbb{K}$ is a field, and for Persistent Laplacians we will explicitly consider only $\mathbb{R}$ for its natural inner-product structure. 

\begin{definition}
    For a simplicial complex $K$, the \textbf{$n$-dimensional chain group} with coefficients in $\mathbb{K}$ is 
    \[
        C_n(K;\mathbb{K}) = \mathrm{span}_{\mathbb{K}}\{n\text{-simplices}\}.
    \]
    When the complex or coefficient system is clear from context, we just use $C_n(K)$ or $C_n$.
\end{definition}

In the setting of persistent Laplacians, this means that the chain groups are vector spaces, where the basis elements are indexed by the simplices. These spaces $C_n(K)$ come with a boundary map $d_n$, which geometrically sends an $n$-simplex to its boundary. The combinatorial description accounts for orientations of the simplices. 

\begin{definition}
    The boundary map $d_n:C_n(K)\to C_{n-1}(K)$ is the linear map defined on basis elements by
    \[
    d_n([v_0,v_1,\dots,v_n]) =\sum_{i=0}^{n}(-1)^i[v_0,v_1,\dots,\widehat{v_i},\dots, v_n],
    \]
    where $\widehat{v_i}$ means that vertex $v_i$ is omitted from the list of vertices in the simplex.
\end{definition}

The chain groups with boundary maps forms a chain complex $(C_\bullet,d_\bullet)$, meaning $d_n\circ d_{n+1}=0$ for all $n$, which is shown in more detail in \cite{HatcherAT}. We often write the chain complex as

\[
    \cdots \xrightarrow[]{d_{n+1}} C_n\xrightarrow[]{d_n}C_{n-1}\xrightarrow[]{d_{n-1}}\cdots \xrightarrow[]{d_2} C_1\xrightarrow[]{d_1}C_0\to 0.
\]

Since $d_n\circ d_{n+1}=0$, we know $\Ima d_{n+1}\subset \ker d_n$, which means the quotient $\ker d_n/\Ima d_{n+1}$ is defined. This is called the homology.

\begin{definition}
    The $n$-th homology of a simplicial complex $K$ with coefficients in $\mathbb{K}$ is 
    \[
    H_n(K;\mathbb{K}) = \ker d_n/ \Ima d_{n+1}.
    \]
    When the coefficient system is clear from context, we just use $H_n(K)$. Elements of $\ker d_n$ are called cycles and elements of $\Ima d_{n+1}$ are called boundaries.
\end{definition}

The standard interpretation of homology is that its rank, called the Betti number $\beta_n$, ``counts holes" in the given dimension, i.e. in dimension $1$ we have loops and dimension $2$ we have voids; in dimension $0$ it counts the number of connected components. 

\begin{example}\label{ex:homology}
    In Figure \ref{fig:square_h1} we show a simplicial complex $K$ with vertices $K_0=\{[1],[2],[3],[4]\}$, edges $K_1=\{[12],[13],[14],[24],[34]\}$, and $2$-simplices $K_2=\{[134]\}$. Then the chains are 
    \begin{align*}
        C_0(K;\mathbb{R}) &= \mathrm{span}_{\mathbb{R}}\{[1],[2],[3],[4]\}\\
        &\cong \mathbb{R}^4,\\
        C_1(K;\mathbb{R}) &= \mathrm{span}_{\mathbb{R}}\{[12],[13],[14],[24],[34]\}\\ 
        &\cong \mathbb{R}^5,\\
        C_2(K;\mathbb{R}) &= \mathrm{span}_{\mathbb{R}}\{[134]\}\\
        &\cong \mathbb{R}^1,\\
        C_3(K;\mathbb{R}) &= \{0\}.
    \end{align*}
    The boundary maps are
    \begin{equation*}
        d_1 = \begin{blockarray}{cccccc}
            & [12] & [13] & [14] & [24] & [34]\\
            \begin{block}{c(ccccc)}
                {[1]} & -1 & -1 & -1 & 0 & 0\\
                {[2]} & 1 & 0 & 0 & -1 & 0\\
                {[3]} & 0 & 1 & 0 & 0 & -1\\
                {[4]} & 0 & 0 & 1 & 1 & 1\\
        \end{block}
        \end{blockarray}\text{, }\qquad
        d_2 = \begin{blockarray}{cc}
        & [134]\\
        \begin{block}{c(c)}
            {[12]} & 0\\
            {[13]} & 1\\
            {[14]} & -1\\
            {[24]} & 0\\
            {[34]} & 1\\
        \end{block}
        \end{blockarray}.
    \end{equation*}

    Then $\ker d_1 = \mathrm{span}_{\mathbb{R}}\{[12]-[14]+[24],[13]-[14]+[34]\}$ and $\Ima d_2 = \mathrm{span}_{\mathbb{R}}\{[13]-[14]+[34]\}$. Then by the definition of homology, $H_1(K;\mathbb{R})\cong \mathrm{span}_{\mathbb{R}}\{[12]-[14]+[24]\}\cong \mathbb{R}^1$, and $\beta_1=1$. Also, $\ker(C_0\to 0)= C_0$ and $\Ima d_1 = \mathrm{span}_{\mathbb{R}}\{[2]-[1],[3]-[1],[4]-[1],[4]-[2],[4]-[3]\}=\mathrm{span}_{\mathbb{R}}\{[2]-[1],[3]-[1],[4]-[1]\}.$ Then $H_0(K;\mathbb{R})\cong \mathbb{R}^1$ and $\beta_0=1$. 
    
    Note that although $[13]-[14]+[34]$ is a cycle, it is also the boundary of $[134],$ so it does not contribute to the homology. The loop $[12]-[14]+[24]$ is a nontrivial cycle that does contribute to homology, so $H_1(K)$ is nontrivial. The complex has one connected component, quantified by $\beta_0=1$.
    
    \begin{figure}
        \centering
        \includegraphics[width=1.80cm]{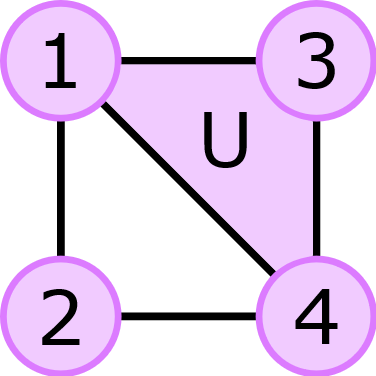}
        \caption{Simplicial complex for Example \ref{ex:homology} and Example \ref{ex:combinatorial}.}
        \label{fig:square_h1}
    \end{figure}
\end{example}

We now turn our attention toward building the Combinatorial Laplacian, where we will fix our coefficients to be $\mathbb{K}=\mathbb{R}$. As each $C_n(K;\mathbb{R})$ is a real vector space, we may give it the standard inner product $\langle \cdot,\cdot\rangle_{n}:C_n\times C_n\to \mathbb{R}$ defined on the basis of $n$-simplices of $C_n$ and extending bilinearly to all of $C_n$, 
\[\langle\sigma,\tau\rangle_n=\begin{cases} 1, \text{ if }\sigma=\tau\\
    0, \text{ if }\sigma\neq\tau
\end{cases}.
\]

Linear maps $d_n:C_n\to C_{n-1}$ have adjoints $d_n^*:C_{n-1}\to C_n$ with respect to these inner products such that $\langle \sigma,d_n\tau\rangle_{n-1} = \langle d_n^*\sigma, \tau\rangle_{n}$. The adjoints are actually the standard coboundary maps. This is because finite-dimensional real inner product spaces are self-dual: $C^n:= \hom (C_n,\mathbb{R})\cong C_n$ through the isomorphism $\Phi:C_n\to C^n$ where $\Phi(\sigma)=\phi_\sigma$, and $\phi_\sigma$ is defined by $\phi_\sigma(\tau)=\langle \sigma,\tau\rangle$. The inverse is $\phi_\sigma\mapsto \sigma$. Under this isomorphism, the coboundary map satisfies the adjoint property, and the adjoint is unique. We can add these adjoints to our chain complex diagram to obtain

\begin{center}        
    \begin{tikzcd}
    C_{n+1}(K) \arrow[r, "d_{n+1}", shift left, red] & C_n(K) \arrow[r, "d_{n}", shift left, blue] \arrow[l, "d_{n+1}^*", shift left, red] & C_{n-1}(K) \arrow[l, "d_{n}^*", shift left, blue]
    \end{tikzcd}
\end{center}

The compositions $\Delta_n^{\text{up}}=d_{n+1}\circ d_{n+1}^*$ and $\Delta_n^{\text{down}}=d_n^*\circ d_n$ are then self-maps of $C_n(K):$

\begin{center}        
    \begin{tikzcd}
    \arrow[loop left, "\Delta_n^{\text{up}}=d_{n+1}\circ d_{n+1}^*", red] C_n(K) \arrow[loop right, "\Delta_n^{\text{down}}=d_{n}^*\circ d_{n}", blue]
    \end{tikzcd}
\end{center}

Moreover, these maps can be added, and this is precisely the combinatorial Laplacian of a simplicial complex.

\begin{definition}
    The \textbf{$n$-th combinatorial Laplacian} $\Delta_n:C_n(K;\mathbb{R})\to C_n(K;\mathbb{R})$ of a simplicial complex $K$ is 
    \begin{align*}
        \Delta_n &= {d_{n+1}d_{n+1}^*} + {d_n^* d_n}\\
        &= \Delta_n^{\text{up}}+\Delta_n^{\text{down}}.
    \end{align*}   
    The maps $\Delta_n^{\text{up}}$ and $\Delta_n^{\text{down}}$ are called the \textbf{up-Laplacian} and \textbf{down-Laplacian}, respectively. The \textbf{spectral gap}, $\lambda_n$, is the least non-zero eigenvalue.
\end{definition}

The combinatorial Laplacian has the following two key properties. 
\begin{theorem} For a simplicial complex $K$, $\ker\Delta_n \cong H_n(K;\mathbb{R})$.    
\end{theorem}
\begin{theorem}For a simplicial complex $K$, $C_n(K;\mathbb{R})\cong \Ima d_n^* \oplus \ker \Delta_n \oplus \Ima d_{n+1}$.
\end{theorem}

There are many proofs of this in the literature. The paper  \cite{lim2020hodge} by Lim explains how these results can be shown for any two linear maps $A$ and $B$ of finite-dimensional real vector spaces such that $AB=0$, which is satisfied by $d_{n}\circ d_{n+1}=0$. 

\begin{example}\label{ex:combinatorial}
    We use the same complex as in Figure \ref{fig:square_h1} and Example \ref{ex:homology}. The $0$-th Laplacian is
    \begin{align*}
        \Delta_0 &= d_1\circ d_1^*\\
        &= \begin{blockarray}{cccccc}
            & [12] & [13] & [14] & [24] & [34]\\
            \begin{block}{c(ccccc)}
                {[1]} & -1 & -1 & -1 & 0 & 0\\
                {[2]} & 1 & 0 & 0 & -1 & 0\\
                {[3]} & 0 & 1 & 0 & 0 & -1\\
                {[4]} & 0 & 0 & 1 & 1 & 1\\
        \end{block}
        \end{blockarray}\begin{blockarray}{ccccc}
            & [1] & [2] & [3] & [4] \\
            \begin{block}{c(cccc)}
                {[12]}  & -1 & 1  & 0  & 0\\
                {[13]}  & -1 & 0  & 1  & 0\\
                {[14]}  & -1 & 0  & 0  & 1\\
                {[24]}  & 0  & -1 & 0  & 1\\
                {[34]}  & 0  & 0  & -1 & 1\\
        \end{block}
        \end{blockarray}\\
        &= \begin{blockarray}{ccccc}
            & [1] & [2] & [3] & [4] \\
            \begin{block}{c(cccc)}
                {[1]} & 3 & -1 & -1 & -1 \\
                {[2]} & -1 & 2 & 0 & -1\\
                {[3]} & -1 & 0 & 2 & -1 \\
                {[4]} & -1 & -1 & -1 & 3 \\
            \end{block}
        \end{blockarray}
    \end{align*}

    A few observations about $\Delta_0$:
    \begin{itemize}
        \item $\Delta_0$ is the graph Laplacian of the underlying graph: diagonal entries are the degree (number of incident edges) of the vertex. Off-diagonal entries are $-1$ if there is an edge between the two vertices and $0$ otherwise.
        \item $\Delta_0$ is symmetric and positive semi-definite
        \item The eigenvalues of $\Delta_0$ are $\{0,2,4,4\}$
        \item The complex has one component, $\beta_0=1$, and the spectral gap is $\lambda_0 = 1$.
    \end{itemize}

    Next, the $1$-st Laplacian is 
\begin{align*}
        \Delta_1 &= d_2\circ d_2^* + d_1^*\circ d_1\\
        &= \begin{blockarray}{cc}
        & [134]\\
        \begin{block}{c(c)}
            {[12]} & 0\\
            {[13]} & 1\\
            {[14]} & -1\\
            {[24]} & 0\\
            {[34]} & 1\\
        \end{block}
        \end{blockarray}\begin{blockarray}{cccccc}
        & [12] & [13] &[14] & [24] & [34]\\
        \begin{block}{c(ccccc)}
            {[134]} & 0 & 1 & -1 & 0 & 1\\
        \end{block}
        \end{blockarray}\\
        &\phantom{=} + \begin{blockarray}{ccccc}
            & [1] & [2] & [3] & [4] \\
            \begin{block}{c(cccc)}
                {[12]}  & -1 & 1  & 0  & 0\\
                {[13]}  & -1 & 0  & 1  & 0\\
                {[14]}  & -1 & 0  & 0  & 1\\
                {[24]}  & 0  & -1 & 0  & 1\\
                {[34]}  & 0  & 0  & -1 & 1\\
        \end{block}
        \end{blockarray}\begin{blockarray}{cccccc}
            & [12] & [13] & [14] & [24] & [34]\\
            \begin{block}{c(ccccc)}
                {[1]} & -1 & -1 & -1 & 0 & 0\\
                {[2]} & 1 & 0 & 0 & -1 & 0\\
                {[3]} & 0 & 1 & 0 & 0 & -1\\
                {[4]} & 0 & 0 & 1 & 1 & 1\\
        \end{block}
        \end{blockarray}\\
        &= \begin{blockarray}{cccccc}
            & [12] & [13] & [14] & [24] & [34]\\
            \begin{block}{c(ccccc)}
                {[12]} & 0 & 0 & 0 & 0 & 0\\
                {[13]} & 0 & 1 & -1 & 0 & 1\\
                {[14]} & 0 & -1 & 1 & 0 & -1\\
                {[24]} & 0 & 0 & 0 & 0 & 0\\
                {[34]} & 0 & 1 & -1 & 0 & 1\\
        \end{block}
        \end{blockarray} + \begin{blockarray}{cccccc}
            & [12] & [13] & [14] & [24] & [34]\\
            \begin{block}{c(ccccc)}
                {[12]} & 2 & 1 & 1 & -1 & 0\\
                {[13]} & 1 & 2 & 1 & 0 & -1\\
                {[14]} & 1 & 1 & 2 & 1 & 1\\
                {[24]} & -1 & 0 & 1 & 2 & 1\\
                {[34]} & 0 & -1 & 1 & 1 & 2\\
        \end{block}
        \end{blockarray}\\
        &= \begin{blockarray}{cccccc}
            & [12] & [13] & [14] & [24] & [34]\\
            \begin{block}{c(ccccc)}
                {[12]} & 2 & 1 & 1 & -1 & 0\\
                {[13]} & 1 & 3 & 0 & 0 & 0\\
                {[14]} & 1 & 0 & 3 & 1 & 0\\
                {[24]} & -1 & 0 & 1 & 2 & 1\\
                {[34]} & 0 & 0 & 0 & 1 & 3\\
        \end{block}
        \end{blockarray}
    \end{align*}

    A few observations about $\Delta_1$:
    \begin{itemize}
        \item $\Delta_1$ is symmetric and positive semi-definite
        \item The eigenvalues of $\Delta_1^{\textrm{up}}$ are $\{0,0,0,0,3\}$ and the eigenvalues of $\Delta_1^{\textrm{down}}$ are $\{0,0,2,4,4\}$.
        \item The nonzero eigenvalues of $\Delta_1^{\textrm{down}}$ are the nonzero eigenvalues of $\Delta_0^{\textrm{up}}$, $\{2,4,4\}$.
        \item The eigenvalues of $\Delta_1$ are $\{0,2,3,4,4\}$.
        \item If $\lambda$ is an eigenvalue of $\Delta_1$, then $\lambda$ is an eigenvalue of $\Delta_1^{\textrm{down}}$ or $\Delta_1^{\textrm{down}}$.
        \item The complex has one nontrivial cycle, $\beta_1 =1$, and the spectral gap is $\lambda_1=2$.
    \end{itemize}
\end{example}

The combinatorial Laplacian generalizes the graph Laplacian, which is often studied by its eigenvalues and eigenvectors. The same is done with the combinatorial Laplacian's eigenvalues. For a particular notion of normalization on a weighted simplicial complex, Horak and Jost \cite{horak2013spectra} were able to bound the eigenvalues of $\Delta_n, \Delta_n^{up},$ and $\Delta_n^{down}$. Grande and Schaub \cite{grande2024disentangling} showed that eigenvectors of combinatorial Laplacians for an $\alpha$-complex capture geometric notions of curl and flow and can be used for topological spectral clustering. Parzanchevski and Rosenthal \cite{parzanchevski2017simplicial} show that combinatorial Laplacian eigenvalues control the convergence of random walks in arbitrary dimensions. Jones and Wei introduced a similar structure for Khovanov homology in knot theory \cite{jones2025khovanov}. This is, the combinatorial Laplacian is of significant independent interest. We seek to study similar aspects of data through a multi-scale lens, and hence we turn toward persistence. 

\subsection{Persistent Laplacian}\label{subsec:persistence}
Persistence is a central notion in TDA, which quantifies how some quality of data changes at different scales in time, space, or some other parameter. There are formalizations of this, including that of persistence modules for persistent homology \cite{zomorodian_computing_2005}  and Laplacian trees for persistent Laplacians \cite{liu_algebraic_2024}. These are often used in proving stability results. To allow a wider audience to access and understand persistent Laplacians, we will use simpler notions around persistence, and remark in places where a more general theory may be of interest. We begin by explaining what we mean by viewing data at different scales.

\begin{definition}
    A \textbf{filtration} of simplicial complexes is a collection of nested simplicial complexes $K=\{K^\alpha\}_{\alpha\in A}$, where $A\subset \mathbb{R}$ and if $a\leq b$ then $K^a\subset K^b$. For each $n$, there are inclusion maps $i_n^{a,b}:K_n^a\hookrightarrow K_n^b$.
\end{definition}

A filtration can also be viewed as functors from a poset category. Among the most common filtrations in TDA are the Vietoris-Rips filtration and the Alpha filtration, which create a filtration of simplicial complexes based on the distances between points in $\mathbb{R}^n$. More details on these filtrations are in Appendix \ref{appendix:complexes}. At each filtration level $K^a$, the simplicial complex $K^a$ has a chain complex $(C_\bullet^a,d_\bullet^a)$, and each inclusion map $i_{n}^{a,b}:K_n^a\hookrightarrow K_n^b$ induces a map $i_{n}^{a,b}:C_n^a\hookrightarrow C_n^b$ on the chain complexes,
\begin{center}
\begin{tikzcd}
C_{n+1}^a \arrow[d, "{i_{n+1}^{a,b}}", hook, shift left] \arrow[rr, "d_{n+1}^a"] &  & C_n^a \arrow[rr, "d_{n}^a", shift left] \arrow[d, "{i_n^{a,b}}", hook, shift left] &  & C_{n-1}^a \arrow[d, "{i_{n-1}^{a,b}}", hook, shift left] \\
C_{n+1}^b \arrow[rr, "d_{n+1}^b"]                                                &  & C_n^b \arrow[rr, "d_n^b"]                                                          &  & C_{n-1}^b                                               
\end{tikzcd}
\end{center}

These inclusions on the chain complex induce maps $i_n^{a,b}:H_n^{a}(K)\to H_n^b(K)$ between the corresponding homologies of $K^a$ and $K^b$. These are not necessarily inclusions, which is measured by persistent homology. 
\begin{definition}
    The \textbf{$n$-th $(a,b)$-persistent homology} with coefficients in $\mathbb{K}$ of a filtration $\{K^\alpha\}_{\alpha \in A}$ of simplicial complexes is
    \[
        H_n^{a,b}(K;\mathbb{K}) = \Ima\left(i_n^{a,b}:H_n^a(K;\mathbb{K})\to H_n^b(K;\mathbb{K})\right).
    \]

    We call $\beta_n^{a,b}=\dim H_n^{a,b}$ the \textbf{$n$-th $(a,b)$-persistent Betti number.} 
\end{definition}

Persistent homology is often studied via the general notion of persistence modules, which are uniquely decomposed into a direct sum of interval modules $\mathbb{K}[a,b)$ and $\mathbb{K}[a,\infty)$ \cite{zomorodian_computing_2005,bubenik2021homological}. Each interval module $\mathbb{K}[a,b)$ corresponds to a cycle that first appears (``born") in $H_n^{a}$ and first becomes a boundary (``dies") in $H_n^b$, and each $\mathbb{K}[a,\infty)$ corresponds to a cycle born in $H_n^a$ and that does not die in filtration. The intervals are visualized in two main ways. First, as a \textbf{persistence diagram}, where a point $(a,b)$ is plotted in $\mathbb{R}^2$ for each $[a,b)$ and $[a,\infty)$ is represented by a point with $x$-coordinate $a$ at a designated space above the rest of the graph. Alternatively, the interval decomposition can be plotted as line segments of a \textbf{barcode}. The segments corresponding to $\mathbb{K}[a,b)$ are placed above the horizontal axis from $x=a$ to $x=b$ and $\mathbb{K}[a,b)$ is denoted by a segment from $x=a$ to the end of the graph. The vertical axis is used only for indexing, sometimes ordered by earliest birth. These representations can be vectorized to form machine learning features in a variety of ways, including persistence landscapes \cite{persistenceLandscapes2015Bubenik} and persistence images \cite{adams_persistence_image}.

Now we use the notions of persistence and filtrations to form persistent topological Laplacians. We will assume $\mathbb{K}=\mathbb{R}$ for the remainder of this article. We will use constructions and arguments previously described in \cite{wang2020PSG, memoli_PL, gulen_et_al:LIPIcs.SoCG.2023.37}. In the same manner as in Subsection \ref{subsec:combinatorial}, we can give each chain group $C_n^a$ an inner product, and each boundary map $d_n^a:C_n^a\to C_{n-1}^a$ has an adjoint $\left(d_n^a\right)^*:C_{n-1}^a\to C_n^a$. 

We would like to construct a persistent Laplacian that is a self-map of a particular chain group, $\Delta_n^{a,b}:C_n^a\to C_n^a$, such that $\ker\Delta_n^{a,b}\cong H_n^{a,b}(K;\mathbb{R})$. We must incorporate information about $K^b$ and $(C_\bullet^b,d_\bullet^b)$. A natural approach is to want some sort of map $C_{n+1}^b\to C_n^a$ capturing the information of $d_{n+1}^b$. It is possible that $\sigma\in C_{n+1}^b$ and $d_{n+1}^b(\sigma)\in C_n^b\setminus C_n^a$, as in example \ref{ex:nontrivial_PL}. 

\begin{figure}
    \centering
    \includegraphics[width=4.0cm]{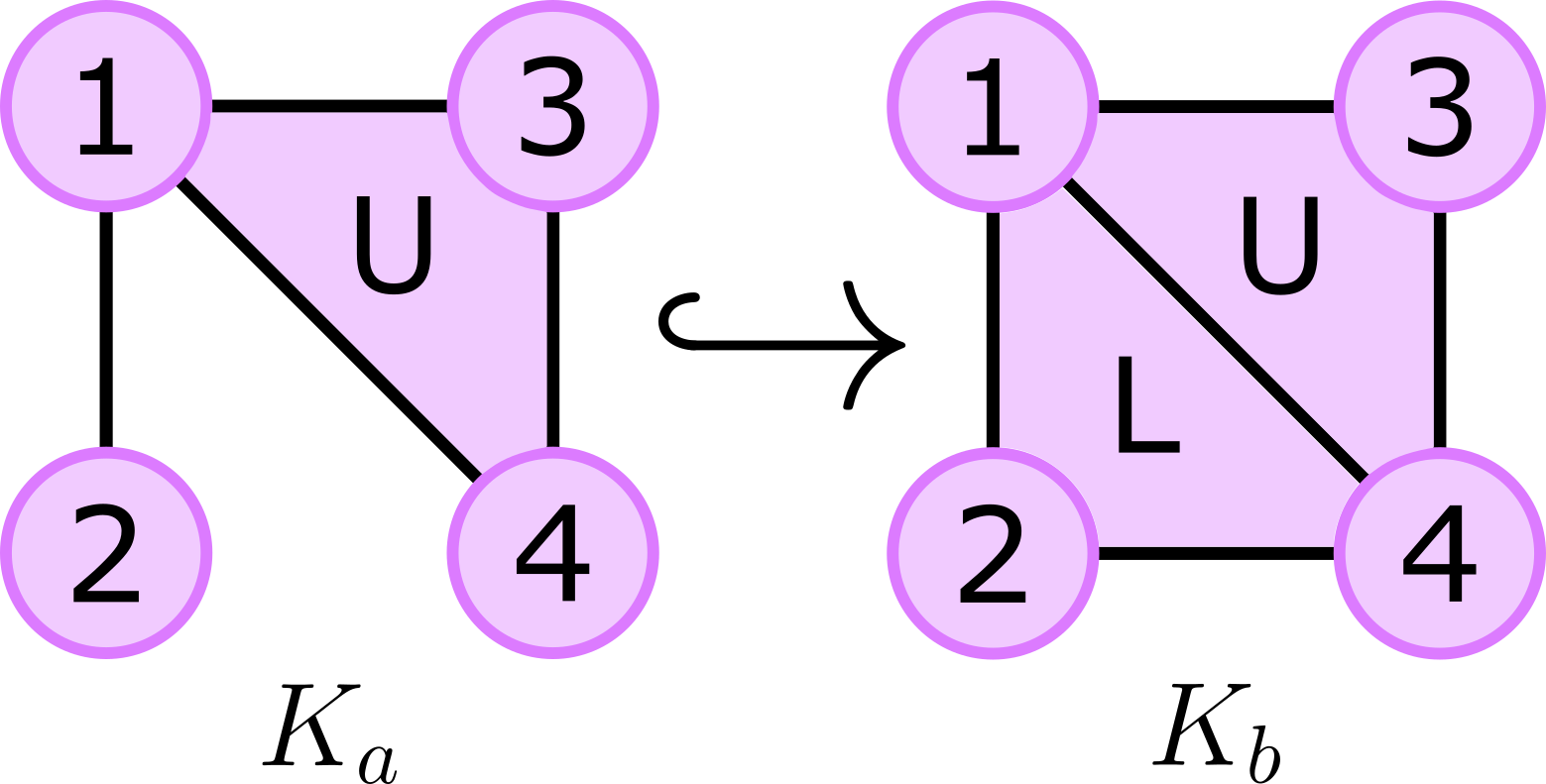}
    \caption{A filtration step of simplicial complexes where restricting $d_n^b$ to $C_n^{a,b}$ is not as simple as removing rows or columns.}
    \label{fig:nontrivial_PL}
\end{figure}

\begin{example}\label{ex:nontrivial_PL}
    Consider the filtration $K_a\subset K_b$ in Figure \ref{fig:nontrivial_PL}. Then $d_2^b([124])=[12]-[14]+[24]\in C_1^b\setminus C_1^a$. 
\end{example}

We must modify $d_{n+1}^b$ in such a way that its image lands in $C_n^a$, in order to obtain a map to $C_n^a$. To do this, we restrict ourselves to the subset of chains in $C_{n+1}^b$ with boundary contained entirely in (the image under inclusion of) $C_n^a$ \cite{gulen_et_al:LIPIcs.SoCG.2023.37}:

\[C_{n+1}^{a,b}=\{\sigma\in C_{n+1}^b | d_{n+1}\sigma \in  i_n^{a,b}(C_n^a)\}.\]

This is an inner-product subspace of $C_{n+1}^b$, but critically a basis of $C_{n+1}^{a,b}$ is not necessarily a restriction of a basis of $C_{n+1}^b$. This will be the source of much computational complexity. We can define the restriction of $d_{n+1}^b$ to $C_{n+1}^{a,b}$ as the persistent boundary map $d_{n+1}^{a,b}:C_{n+1}^{a,b}\to C_{n}^a$. As $C_{n+1}^{a,b}$ is a subspace of $C_{n}^b$, it inherits the existence of an inner product, but due to the mismatch of bases it may be represented differently. Still, there is an adjoint $\left(d_{n+1}^{a,b}\right)^*:C_n^a\to C_{n+1}^{a,b}$. These ingredients are placed in context in the following diagram:

\begin{center}
\begin{tikzcd}[column sep=huge, row sep=large]
C_{n+1}^a \arrow[dd, "{i_{n+1}^{a,b}}", hook, shift left] \arrow[rr, "d_{n+1}^a"] &                                                                                                      & C_n^a \arrow[rr, "d_{n}^a", shift left]\arrow[rr, "d_{n}^a", shift left, blue] \arrow[ld, "\left(d_{n+1}^{a,b}\right)^*",red] \arrow[dd, "{i_n^{a,b}}", hook, shift left] &  & C_{n-1}^a \arrow[ll, "\left(d_{n}^a\right)^*", shift left, blue] \arrow[ll, "\left(d_{n}^a\right)^*", shift left]\arrow[ll, "\left(d_{n}^a\right)^*", shift left, blue]  \arrow[dd, "{i_{n-1}^{a,b}}", hook, shift left] \\
                                  & {\textcolor{red}{C_{n+1}^{a,b}}}  \arrow[ld, "{\iota_{n+1}^{a,b}}"', hook] \arrow[ru, "{d_{n+1}^{a,b}}", shift left=2, red] \arrow[ru, "{d_{n+1}^{a,b}}", shift left=2] \arrow[ru, "{d_{n+1}^{a,b}}", shift left=2, red] &                                                                                                                                  &  &                                                                                                              \\
C_{n+1}^b \arrow[rr, "d_{n+1}^b"]                                                 &                                                                                                      & C_n^b \arrow[rr, "d_n^b"]                                                                                                        &  & C_{n-1}^b                                                                                                   
\end{tikzcd}
\end{center} 

The maps $d_n^a$ and $\left(d_n^a\right)^*$ are of the same construction as in the standard combinatorial Laplacian of the simplicial complex $C_n^a$. This peculiarity is addressed in \cite{gulen_et_al:LIPIcs.SoCG.2023.37}, where the authors restrict $d_n^a$ to a set $C_{n-1}^{b,a}\subset C_{n-1}^a$, which is analogous in definition to $C_{n+1}^{a,b}$ for the down-Laplacian. The authors show that $C_{n-1}^{b,a}=C_{n-1}^a,$ and hence the necessary restriction of $\left(d_n^a\right)^*$ is not a restriction at all. 

Now in the same way as before, compositions $\Delta_n^{a,\text{up}}=d_{n+1}^{a,b}\circ \left(d_{n+1}^{a,b}\right)^*$ and $\Delta_n^{a,\text{down}}=\left(d_n^a\right)^*\circ d_n^a$ are self-maps of $C_n^a:$

\begin{center}        
    \begin{tikzcd}
    \arrow[loop left, "\Delta_{n,\text{up}}^{a,b}=d_{n+1}^{a,b}\circ \left(d_{n+1}^{a,b}\right)^*", red] C_n^a \arrow[loop right, "\Delta_{n,\text{down}}^{a}=\left(d_n^a\right)^*\circ d_n^a", blue]
    \end{tikzcd},
\end{center}

which can again be added together. This brings us to the main definition of this work.

\begin{definition}
    The \textbf{$n$-th $(a,b)$-persistent Laplacian} of a filtration $\{K^\alpha\}_{\alpha\in A}$ of simplicial complexes is $\Delta_n^{a,b}:C_n^a\to C_n^a$, given by
    \begin{align*}
        \Delta_n^{a,b}&=d_{n+1}^{a,b}\circ\left(d_{n+1}^{a,b}\right)^* + d_n^a\left(d_n^a\right)^*\\
        &= \Delta_{n,\text{up}}^{a,b} + \Delta_{n,\text{down}}^{a}.  
    \end{align*}
\end{definition}

This has the following key properties.

\begin{theorem} For a filtration $K=\{K_a\}_{a\in A}$ of finite simplicial complexes and a persistent Laplacian $\Delta_n^{a,b}$: 
\begin{enumerate}
    \item $\Delta_n^{a,b}$ is self-adjoint
    \item $\ker \Delta_n^{a,b}\cong H_n^{a,b}(K;\mathbb{R})$
    \item $C_n^a(K;\mathbb{R})= \Ima\left(d_n^a\right)^*\oplus \ker \Delta_n^{a,b}\oplus \Ima d_{n+1}^{a,b}$
    \item (Example 2.3.3 in \cite{memoli_PL}) $\Delta_n^{a,a}=\Delta_n^a$.
\end{enumerate}
\end{theorem}

Property 1 is immediate from the definition of $\Delta_n^{a,b}$. Property 2 is shown in the proof of Theorem 2.7 in \cite{memoli_PL}. Although Theorem 2.7 itself only asserts the dimensions, the authors prove this by showing this isomorphism is a direct application of Theorems 5.2 and 5.3 in \cite{lim2020hodge}. Property 3 is another straightforward application of Theorem 5.2 in \cite{lim2020hodge}, and was stated in \cite{Lieutier_PersistentHarmonicForms2014}.

\begin{lemma} (Theorem 5.2.9 in \cite{lim2020hodge})
Let $A\in \mathbb{R}^{m\times n}$ and $B\in\mathbb{R}^{n\times p}$ with $AB=0$. Then $\mathbb{R}^n=\Ima (A^*)\oplus \ker(A^*A + BB^*)\oplus \Ima (B)$.
\end{lemma}

Applying this to $A=d_n^a, B=d_{n+1}^{a,b}$ gives the persistent Hodge decomposition. A clearly explained example computation of a persistent Laplacian must be delayed until we have established the tools of subsection \ref{subsec:computation}, but here we can show the code to produce a persistent Laplacian and its eigenvalues for the filtration in Figure \ref{fig:nontrivial_PL}:

\begin{lstlisting}
import gudhi as gd
import petls

st = gd.SimplexTree()
st.insert([1, 3, 4], filtration=0.0)
st.insert([1, 2], filtration=0.0,
st.insert([2, 4], filtration=0.0)
st.insert([1, 2, 4], filtration=1.0)

complex = petls.Complex(simplex_tree=st)
print(complex.spectra(dim=1, a=0, b=1))
\end{lstlisting}

This prints eigenvalues $[2, 2, 4, 4, 4]$.

One notable exception to the standard framework for constructing PTLs is the Persistent Sheaf Laplacian (PSL)\cite{wei_sheaf_2025}, which can encode domain-informed information. Its construction mirrors many aspects of the PTL for simplicial complexes. A cellular sheaf attaches vector spaces and linear maps to a cell complex. In our case, we will attach copies of $\mathbb{R}$ to the simplices of a simplicial complex. These vector spaces and the restriction maps can be used to model opinion dynamics \cite{hansen2019toward}, sensor networks \cite{curry2013sheaves}, and molecular interactions \cite{wei_sheaf_2025}, mainly through sheaf cohomology and persistent or spectral variations. Appendix \ref{appendix:sheaf} gives an overview of formalizing these notions. More thorough treatments can be found in \cite{curry2013sheaves,wei_sheaf_2025,hansen2019toward}.

The meaning of the zero-eigenvalues of a persistent Laplacian is well understood as the persistent Betti numbers, but there is no explicit theoretical characterization of the nonzero eigenvalues of persistent Laplacians in terms of the geometry of the complex, so explicit interpretation of values is an important open challenge. Cheeger-type inequalities have been established to bound the least nonzero eigenvalues of persistent up-Laplacians with geometric and combinatorial structure for $\Delta_0^{a,b}$ \cite{memoli_PL} and for $\Delta_i^{a,b}$, where $i>0$ \cite{dong2024faster}. Various stability results have been found for the eigenvalues of PTLs, supporting their use as a tame descriptor \cite{memoli_PL,dong2024faster,liu_algebraic_2024}. There are several examples where changes in the nonzero eigenvalues of $\Delta_i^{a,b}$ correspond to structure not detected by persistent homology \cite{wei_Survey,wang2020PSG}. Finally, ablation studies have demonstrated that the persistent Laplacian contributes more predictive information than non-persistent combinatorial Laplacians alone \cite{bhusal2024persistent}. These observations contribute to the theme that PTL eigenvalues detect meaningful changes beyond homology, persistent homology, and combinatorial Laplacians, even if we must currently rely on qualitative descriptions or outsource the interpretation to machine learning methods. 

After stability and interpretation, the next natural question is the use of real coefficients. In computational topology, fast computation is often deeply related to the use of coefficients in $\mathbb{Z}_2$, so it is natural to consider removing the requirement for $\mathbb{R}$ coefficients as a route to improving (persistent) Laplacian computations. In fact, $\mathbb{R}$ is not required: one can take coefficients in the inner product space $\mathbb{C}$, although this does not offer direct benefit, since $\Delta_n$ and its eigenvalues will only have real coefficients. We could also use $\mathbb{Q}$, where we lose the fact that the eigenvalues are in the same field as our coefficients. Even further, one can relax the notion of an inner product to a symmetric bilinear form, allowing us to consider finite fields $\mathbb{Z}_p$. We retain definitions of adjoints and a Laplacian operator; however, we sometimes lose important theoretical guarantees related to the Hodge decomposition and harmonic representation as classified in \cite{catanzaro2023harmonic}. This notion is useful for studying harmonic representatives, but cannot be used for a fixed prime order such as $\mathbb{Z}_2$. Moreover, eigenvalues of these Laplacians are roots of polynomials over a finite field $\mathbb{F_q}$, which lie in a field extension $\mathbb{F_{q^n}}$, rather than $\mathbb{R}$. This restricts the use of standard ways of analyzing Laplacian eigenvalues. It is therefore most natural to consider real coefficients.

\subsection{Computation of persistent up-Laplacian}\label{subsec:computation}
Since $\Delta_{n,\mathrm{down}}^{a,b}$ is a straightforward and efficient matrix multiplication, the computational literature focuses on up-Laplacians, which we review here. We adapt the notations of the other articles to match this work.

The introduction of the persistent Laplacian was by Wang et al. in \cite{wang2020PSG}, and the first software implementation was HERMES by Wang et. al in \cite{wang2021hermes}. HERMES was built specifically for Rips and Alpha filtrations and persistent Laplacians in dimensions $0,1$, and $2$. 
HERMES constructed the map $d_{n+1}^{a,b}$, expressed as a matrix $B_{n+1}^{a,b}$, by first constructing a projection matrix $\mathbb{P}^{a,b}_{n+1}:C_{n+1}^{b}\to C_{n+1}^{a,b}$. They used a gauge method from \cite{zhao_3d_2019} for efficiency in both space and time in constructing this projection. The standard projection matrix $P_{n}^{a}:C_n^b\to C_n^a$ was also constructed. Then $B_{n+1}^{a,b}=P_n^a B_{n+1}^b \left(\mathbb{P}_{n+1}^{a,b}\right)^T$. The persistent Laplacian matrix was then assembled following the definition, $L_{n}^{a,b}=B_{n+1}^{a,b}\left(B_{n+1}^{a,b}\right)^T + \left(B_{n}^{a}\right)^TB_{n}^{a}$.
$O(\left(n_{n+1}^b\right)^3)$.

The authors in \cite{memoli_PL} propose multiple new algorithms for computing the persistent Laplacian of weighted simplicial complexes. We modify their notations to ours, and reduce to the unweighted case. The first algorithm of \cite{memoli_PL} does a column reduction on the rows of $B_{n+1}^b$ that are not in $B_{n+1}^a$ to find $Y$ and setting $R_{n+1}^b = B_{n+1}^b Y$. Then $R_{n+1}^b$ is column reduced (the index of each column's lowest nonzero element is unique). Then $B_{n+1}^{a,b}$ is set as a certain submatrix of $B_{n+1}^{b}Y$, $Z$ is computed as a basis for $C_{n+1}^{a,b}$, and $\Delta_{n,\mathrm{up}}^{a,b}=B_{n+1}^{a,b}\left(Z^TZ\right)^{-1}\left(B_{n+1}^{a,b}\right)^T$.

The second algorithm of \cite{memoli_PL} uses a Schur complement to compute $\Delta_{n,\mathrm{up}}^{a,b}$ without computing $B_{n+1}^{a,b}$. For a square block matrix $M=\begin{pmatrix}
    A & B\\
    C & D
\end{pmatrix}$, where $D$ is square, the Schur complement of $D$ in $M$ is $M/D=A-BD^\dagger C$, where $D^\dagger$ is the Moore-Penrose pseudoinverse of $D$. It is shown that $\Delta_{n,\mathrm{up}}^{a,b}$ can be computed as the particular Schur complement of $\Delta_{n,\mathrm{up}}^b$ where $A = \Delta_{n,\mathrm{up}}^a$ is the top left corner of the matrix. The details are in Algorithm \ref{alg:PL_MWW_2}. For a matrix $X$, let $X(a:b,c:d)$ denote the submatrix of $X$ consisting of rows $a$ through $b$ and columns $c$ through $d$, inclusive.

\begin{algorithm}
 \caption{Persistent Laplacian: Algorithm 2 of \cite{memoli_PL}.}
 \label{alg:PL_MWW_2}
 \begin{algorithmic}[1]
 \STATE \textbf{Data:} $B_n^a,B_{n+1}^b$
 \STATE \textbf{Result:} $\Delta_n^{a,b}$
  \STATE Compute $\Delta_{n,\mathrm{down}}^a$ from $B_n^a$
 \STATE Compute $\Delta_{n,\mathrm{up}}^b$ from $B_{n+1}^b$
 \IF{$n_n^a==n_n^b$}
 \RETURN $ \Delta_{n,\mathrm{up}}^b+\Delta_{n,\mathrm{down}}^a$
 \ENDIF
 \STATE A = $\Delta_{n,\mathrm{up}}^b(0:n_n^a,0:n_n^a)$
 \STATE B = $\Delta_{n,\mathrm{up}}^b(0:n_n^a,n_n^a+1:n_n^b)$
 \STATE C = $B^T$
 \STATE D = $\Delta_{n,\mathrm{up}}^b(n_n^a+1:n_n^b,n_n^a+1:n_n^b)$
\STATE $\Delta_{n,\mathrm{up}}^{a,b} = A-BD^{-1}C$ 
 \STATE{\RETURN $\Delta_{n,\mathrm{up}}^{a,b}+\Delta_{n,\mathrm{down}}^a$}
\end{algorithmic}
\end{algorithm}

An algorithm to compute $\Delta_{n,\mathrm{up}}^{a,b}$ for a specific class of complexes was introduced in \cite{dong2024faster}. The complex is assumed to be non-branching, meaning each simplex has at most $2$ co-faces. This includes $\alpha$-complexes and cubical complexes, so this restriction is fairly reasonable for practical applications. The complex is also required to have an orientation compatibility condition, where if $\sigma$ is an $n$-simplex simplex that is a face of exactly $2$ $(n+1)$-simplices, the $(n+1)$-simplices have opposite orientation. The algorithm rearranges columns of $B_{n+1}^{b}$ based on a graph or hypergraph constructed from the rows of $B_{n+1}^{b}$ corresponding to simplices in $K_b$ but not in $K_a$.

\section{Software design}
\label{sec:design}

\subsection{Overview}

Researchers interested in persistent Laplacians may want a descriptor of a dataset, they may want to test the efficiency of the algorithm for computing $\Delta_{n,\textrm{up}}^{a,b}$, or they may want to show evidence of a conjecture about the interpretation of persistent Laplacian eigenvalues. Because there are ongoing developments for various aspects of computing PTLs, it is important that they are straightforward to interchange as the literature grows. The Persistent Homology Algorithms Toolbox (PHAT) \cite{bauer2017phat} has filled this role in the persistent Homology literature.

In addition to the philosophical need for a flexible framework, there is a technical factor motivating the need for a framework: the asymptotic time complexity of one up-Laplacian algorithm may be best for one type of complex, while another up-Laplacian algorithm may have the best time complexity for another type of complex \cite{memoli_PL} \cite{dong2024faster}. It is therefore critical that a general tool for computing assorted PTLs does not fix a "best" algorithm.

To maximize efficiency, computations are written in C++ using the fast open source library Eigen \cite{eigenweb}. A Python interface is added through pybind11 \cite{pybind11}, so that it can be used by a wider audience.

The software implementations for both C++ and Python allow one to construct a custom data structure named \texttt{Complex} from a point cloud (Vietoris-Rips or Alpha), a directed graph, a Gudhi simplex tree, or a sequence of boundary matrices and filtrations \cite{gudhi:urm}.

In addition to filtered boundary matrices, the \texttt{Complex} stores the up-Laplacian algorithm and the eigenvalue algorithm as function wrappers. A researcher who develops a new up-Laplacian algorithm needs only to write a single function that takes one filtered boundary matrix and the filtration values, then store a reference to that function in the \texttt{Complex}. The default is the Schur complement algorithm of \cite{memoli_PL}. Similarly, the eigenvalue algorithms can be set to one of several defaults or an external program. The \texttt{Complex} allows one to separately compute $\Delta_{n,\mathrm{up}}^{a,b}$, $\Delta_{n,\mathrm{down}}^{a}$, or $\Delta_{n}^{a,b}$ as desired, as well as directly compute the eigenvalues or eigenvalue-eigenvector pairs of $\Delta_{n}^{a,b}$.

This contributes significantly to the software's role in the development of novel algorithms. One only needs to write a single function that computes the up-Laplacian from a boundary matrix, without needing to address the down-Laplacian, the simplicial complex and boundary maps, or the eigenvalue computation. 

\subsection{Persistent sheaf Laplacian}
A persistent sheaf Laplacian (PSL), as introduced in \cite{wei_sheaf_2025}, is a PTL defined for a cellular sheaf on a simplicial complex. A more detailed description of cellular sheaves is in Appendix \ref{appendix:sheaf}. The PSL is distinct from other PTL construction, largely because the boundary matrix of a standard simplicial complex has coefficients in $\{-1,0,1\}$, but this is not true for the cellular sheaf. 

A cellular sheaf with real stalks has a boundary map that is a linear combination of real-valued linear maps, where each restriction map is a real-valued matrix that can take essentially any real value, as long as they assemble to respect the composition and identity rules. This means that although the boundary map of a simplicial complex can be stored as an integer-valued matrix, a cellular sheaf's boundary matrix must be stored as a real-valued matrix, which is less efficient for both storage and basic operations (e.g., the matrix multiplication for $\Delta_{n,\mathrm{down}}^{a}$).

To create the PSL, we currently restrict ourselves to cellular sheaves with stalks $\mathbb{R}^1$. The cellular sheaf is represented as a \texttt{sheaf\_simplex\_tree}, which wraps a Gudhi Simplex Tree \cite{gudhi:urm} and a function \texttt{restriction}. Since a cellular sheaf is a mathematical structure that can be used to encode many types of rich information, largely conveyed by the restriction function, the restriction function must be written by the user. This function has access to the entire sheaf simplex tree, including optional auxiliary data, and needs to only be explicitly defined in the code for faces of codimension $1$. In other words, only the restriction between a vertex and edge, an edge and $2$-simplex, etc., need to be explicitly implemented. For interpreting the meaning of this computation, it is important to define this restriction in a way that satisfies the composition and identity requirements, despite the fact that there are no technical barriers to the matrix computations that follow.  

Since the only computational difference between the PSL and the standard PTL is the matrix storage, the natural approach would be to define each as an instance of a C++ template. This results in prohibitively long compilation times. Therefore, the PSL is implemented separately, but identically to the standard PL, except for the storage type.

This implementation of a PSL expands on previous ones by allowing for arbitrarily-written restriction functions, as well as most of the applicable performance and flexibility enhancements that the standard Persistent Laplacian implementation provide. Moreover, this construction is primed for generalization to stalks in $\mathbb{R}^d$ for $d>1$, because the process of computing boundary matrices from the restriction map could be modified to produce multiple rows and columns for a single simplex.

\section{Computation}
\label{sec:computation}

We now describe details about how the class \texttt{Complex} computes a persistent topological Laplacian and its eigenvalues. We first describe the default process, which motivates the development of some variations.

\subsection{The default pipeline}

First, a set of filtered boundary matrices is constructed. This can be done externally by the user or by one of several methods in the library that build filtered simplicial complexes. For example, an interface is given to extract these filtered boundary matrices from the C++ program Ripser \cite{bauer2021ripser}, which efficiently builds a Rips filtration from a point cloud or distance matrix. These filtered boundary matrices are stored as sparse matrices with integer coefficients.

The \texttt{Complex} permanently stores only $B_{n}$, the boundary matrices corresponding to the largest filtration value, and we compute other boundary matrices only as submatrices of $B_{n}$. Each nonzero entry in $B_{n}$ is $-1$ or $1$, so we can use integer storage, even though we will eventually use floating point operations.

Note that each column of $B_{n}$ has exactly $(n+1)$ nonzero elements, and so will generally be sparse. Then computing $\Delta_{n,\mathrm{down}}^{a,b}=\left(B_{n}^a\right)^TB_{n}^a$ can be done efficiently using sparse matrix multiplication. This motivates storing $B_{n}$ and each $B_{n}^{a}$ as sparse matrices.

After the filtered boundary matrices are stored, one of several ``spectra" functions can be used to obtain the eigenvalues of Persistent Laplacians. For example, \texttt{spectra(n,a,b)} returns a vector containing the eigenvalues of $\Delta_{n}^{a,b}.$ It does this by computing the up Laplacian, down Laplacian, adding them together (as appropriate). The down-Laplacian is computed and stored as a sparse matrix using sparse-sparse multiplication. The up-Laplacian is constructed by calling a wrapper function, which defaults to the Schur complement algorithm of \cite{memoli_PL}. Sparse and self-adjoint storage and computations are performed where possible. The Schur complement includes the computation of $D^{\dagger}C$, which is done by first calculating a robust Cholesky LDL decomposition, specialized for the positive semidefinite form of $D$, and then numerically computing an approximate solution of $Dx = C$. Single precision floating point computations are used, rather than double precision, because the performance implications. We found that there is negligible difference in the resulting eigenvalues, but the computation time differs by about a factor of $2$, making it worthwhile to use a single precision. 

Once the up-Laplacian is computed, we add the down-Laplacian and compute the eigenvalues. The eigenvalue computation is also done by calling a replaceable wrapper function, which defaults to a combination of Eigen's \texttt{SelfAdjointEigenSolver} and, in the event that the Laplacian matrix is diagonal, the default wrapper will skip the traditional numerical method and just return the diagonal. This occurs in the final Laplacians of a non-thresholded Rips filtration. This is what MATLAB does, and it saves a great deal of computation time. By assuming the matrix is self-adjoint, the eigenvalues are known to be real, and this is much more efficient than a general eigenvalue algorithm. When only the eigenvalues are wanted, we are able to calculate them without the eigenvectors to improve efficiency. The eigenvalue algorithm performs a tridiagonal decomposition via Householder reflections, then it is diagonalized with an iterative, symmetric and shifted QR method \cite{eigenweb}.

\subsection{Comparison with other methods}

With this implementation described, its efficiency can now be compared with others in the literature, namely the C++ software HERMES \cite{wang2021hermes} and the MATLAB implementation of \cite{memoli_PL}. Since HERMES requires either an alpha or rips filtration and the MATLAB implementation was distributed with a test case using a rips complex, we first consider a rips filtration. All computational experiments were performed on a single core.

We sample $30$ points from the unit sphere and build a rips filtration with maximum simplex dimension $3$ and no distance cutoff. HERMES used its own rips computation, PETLS used its integrated modification of Ripser, and the MATLAB implementation read boundary matrix files output by Dionysus \cite{morozov2017dionysus}. To compute the eigenvalues, HERMES uses the MATLAB \texttt{eigs} function to compute the smallest $100$ eigenvalues (or all, if there are fewer than $100$), and the MATLAB implementation of \cite{memoli_PL} uses the MATLAB \texttt{eig} function to compute all eigenvalues. When the matrix is diagonal, all methods simplify to extracting the diagonal as the eigenvalues.

We compare this with several distinct ways to use PETLS. First, the standard C++ usage. Second, the standard Python usage. Third, we use the C++ version configured to use OpenBLAS and LAPACK \cite{OpenBLAS}, which Eigen supports as a way to dramatically speed up some linear algebra computations. Fourth, for the eigenvalue computation we use the C++ library Spectra \cite{spectra}, which employs an implicitly restarted Lanczos/Arnoldi method similar to ARPACK. In this example, we use Spectra to extract the $10$ smallest eigenvalues, and if the method from Spectra fails to converge after $1000$ iterations, it falls back to the \texttt{SelfAdjointEigenSolver} from Eigen. Note that this failure in convergence happens \textit{often}. This is largely due to the combination of searching for the smallest eigenvalues and the presence of zero-eigenvalues, which disrupts the numerical algorithm. For example, the dimension $0$ Laplacian of most filtrations will often have a high multiplicity of $0$ as an eigenvalue until the complex is connected. This particular step has a trivial contribution to the overall computation cost when using the \texttt{SelfAdjointEigenSolver}. Here again we can skip all of these steps if the matrix is diagonal. Fifth, we call this same algorithm from Python. Note that Eigen also supports the use of the Intel Math Kernel Library (MKL), which may further improve performance, but this is again not portable and also not freely available, so we do not analyze it - although someone seeking to truly maximize performance may benefit from its use.

The same point cloud samples were used for all software configurations. HERMES computes strictly $\Delta_{d}^{a,a + \delta}$ for $d\in\{0,1,2\}$ and $a\in A$, where $A$ is a finite set of positive real numbers. To fairly compare the other software against HERMES, we compute $\Delta_d^{a,a+0.2}$ and its eigenvalues for $d\in D=\{0,1,2\}$ and $a\in A=\{0.0, 0.2, 0.4,\dots, 1.8, 2.0\}$.  For each $d\in D$ and $a\in A$, we track the time taken to compute the matrix of $\Delta_d^{a,a+0.2}$ and its eigenvalues. The number of $1$-simplices in the final filtration step is $435$ and the number of $2$-simplices is $4060$. We repeat these computations for $100$ different random samples of points from the sphere and report the average resulting times.

\begin{figure}[htpb]
    \centering
    \begin{subfigure}{5.5in}
        \centering
      \includegraphics[width=5.5in]{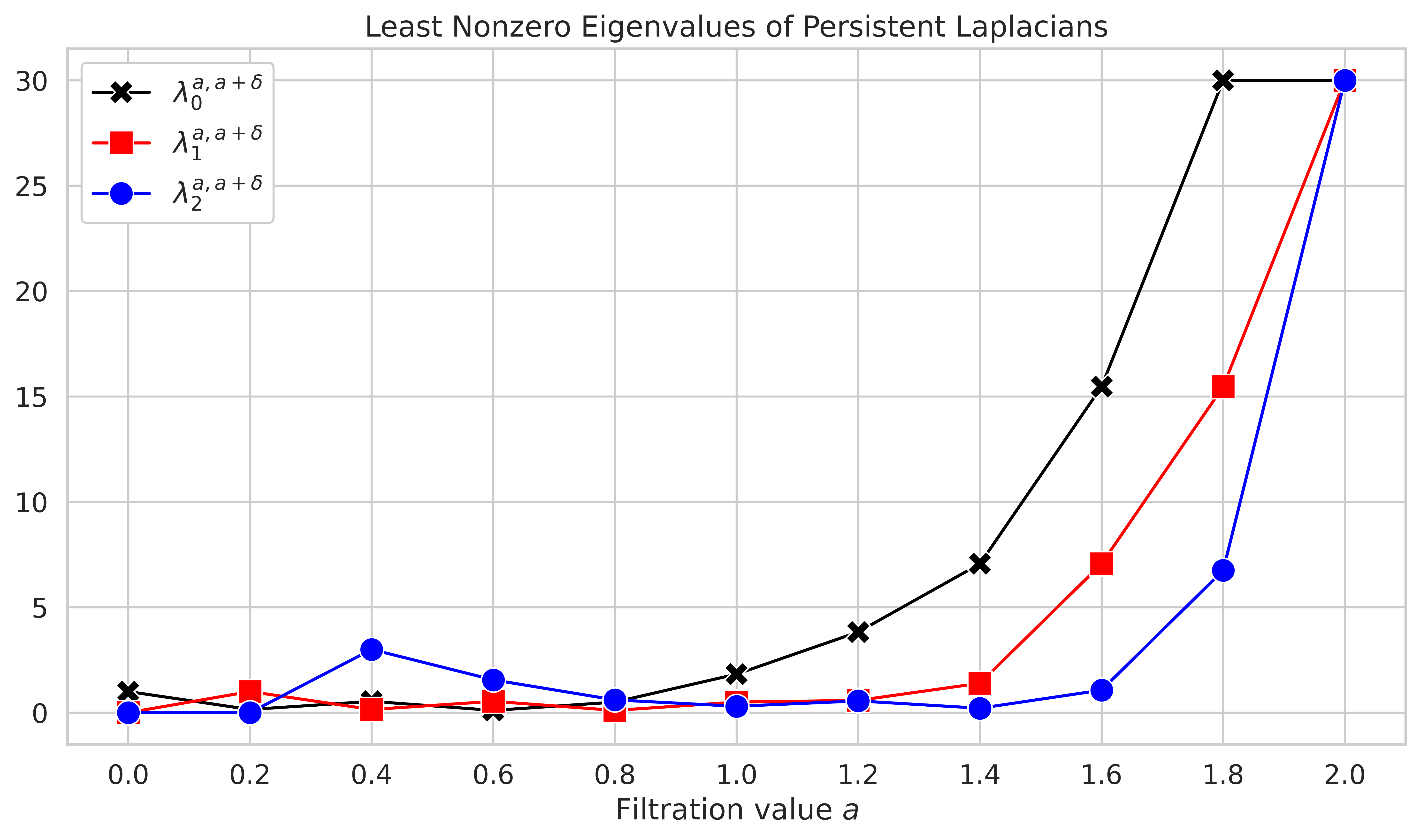}
        \caption{}
    \end{subfigure}
\caption{The least nonzero eigenvalues of persistent Laplacian matrices of a Rips complex from $30$ points sampled from the unit sphere. Values are shown from one individual point sample and are not averaged.}
\label{fig:curve_sphere_rips}
\end{figure}

Figure \ref{fig:curve_sphere_rips} shows the least nonzero eigenvalue of the persistent Laplacian matrices changing as the filtration increases. Figure \ref{fig:times_sphere_rips_petls}(a) shows how the cumulative computation time grows for three of the different PETLS configurations. Note that some of the smaller computation times are recorded as $0$ when they are below the precision of the clock used to measure the speed. The cumulative computation time for dimensions $0$ and $1$ is under $0.05s$ for each PETLS configuration, so we focus on dimension $2$. First, observe that the eigenvalue computations (dashed) take much longer than the matrix computations (solid) for the standard C++ and Python configurations, although for the Python configuration using Spectra the eigenvalue computation is slightly faster than the matrix computation. The Python matrix computations take the same time as the C++ matrix, indicating little overhead cost. The Python eigenvalue computation does take significantly longer in this example, but this is still generally much faster than HERMES and the MATLAB persistent Laplacian. 

\begin{figure}[htpb]
    \centering
    \begin{subfigure}{5.5in}
       \centering
        \includegraphics[width=5.5in]{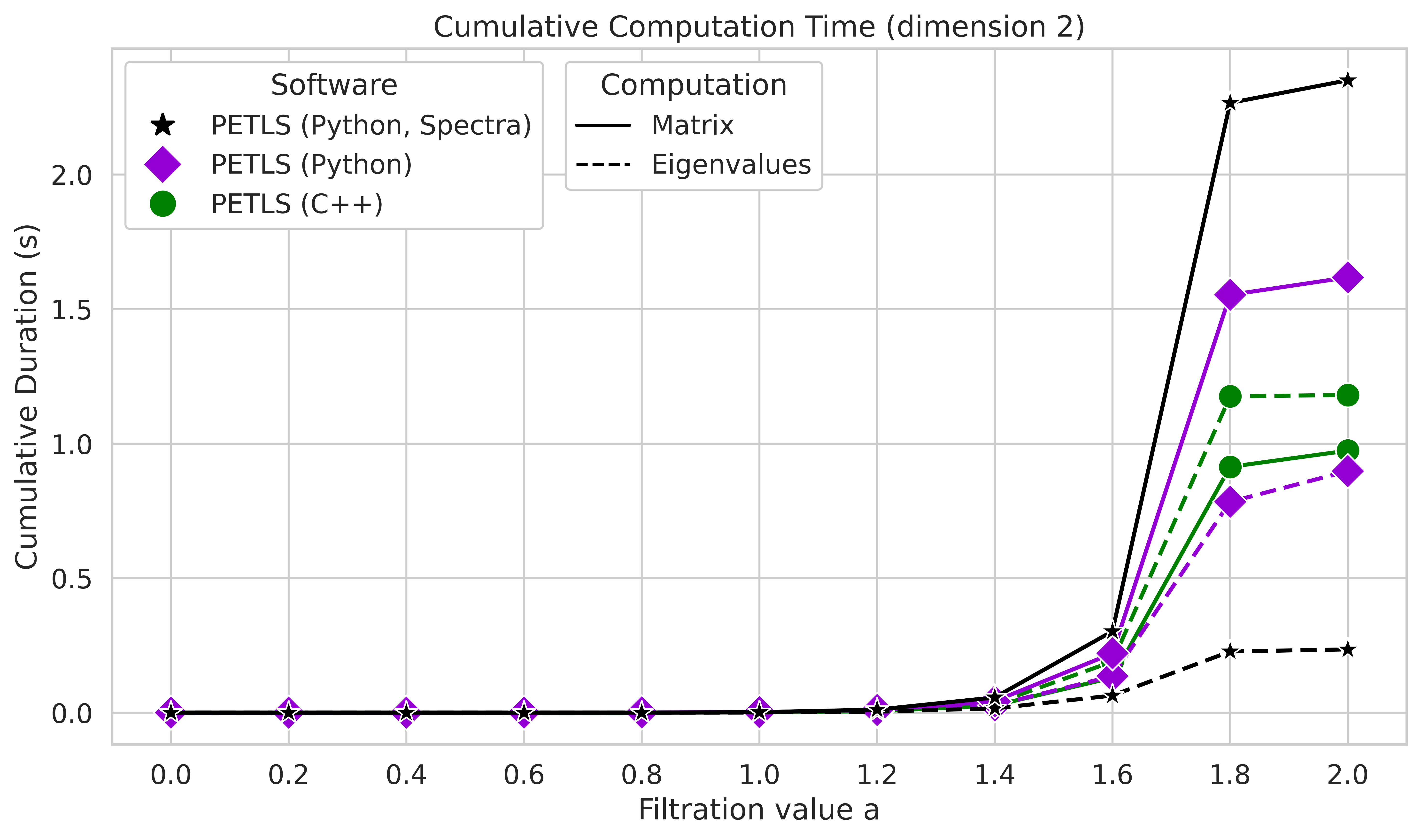}
        \caption{}
    \end{subfigure}
    \begin{subfigure}{5.5in}
        \centering
        \includegraphics[width=5.5in]{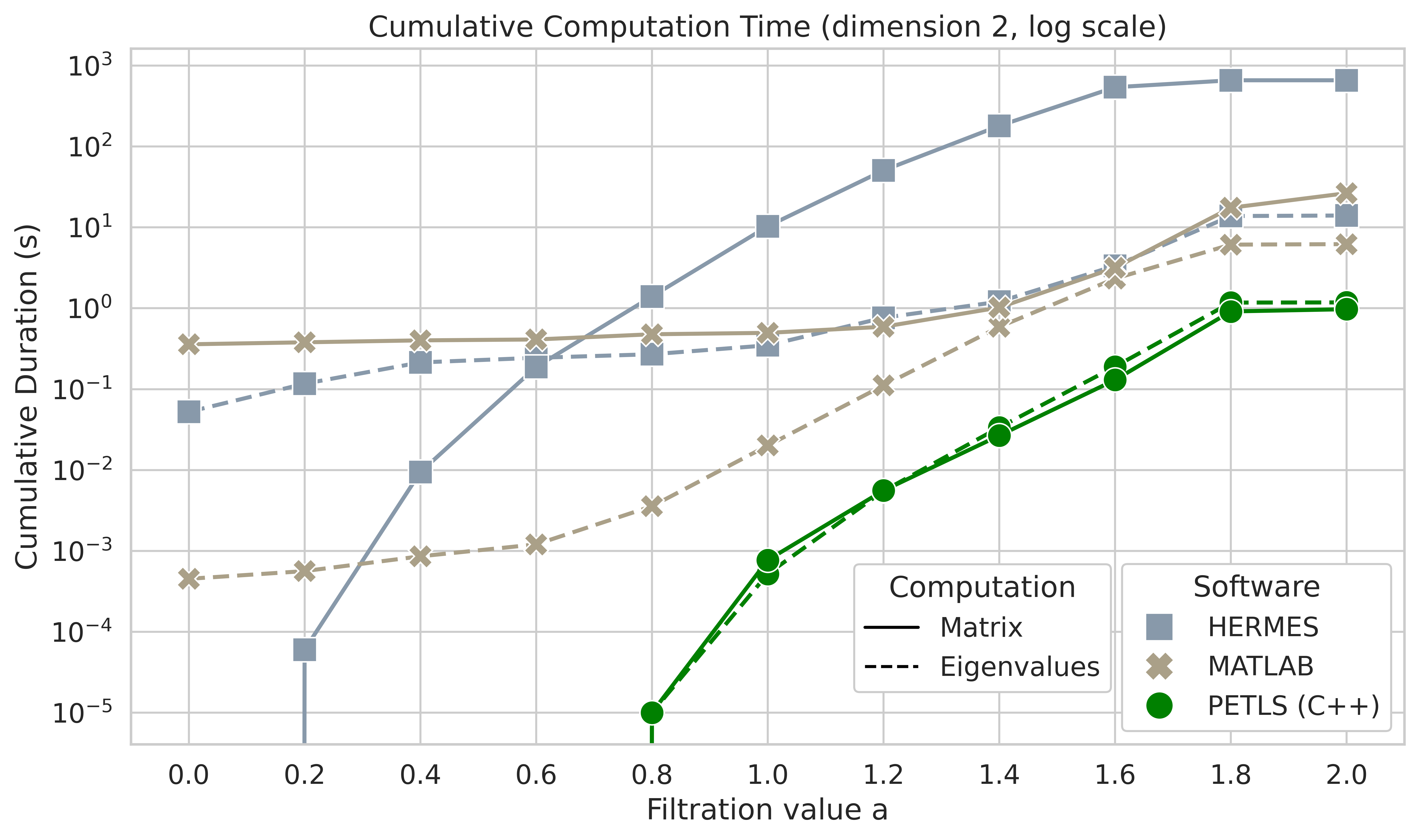}
        \caption{}
    \end{subfigure}
\caption{Computation time using several PETLS configurations to compute persistent Laplacian matrix and eigenvalues of a Rips complex from $30$ points sampled from the unit sphere. Results are averaged over $100$ replicates.}
\label{fig:times_sphere_rips_petls}
\end{figure}

In Figure \ref{fig:times_sphere_rips_petls}(b) we compare computation times for the Python with Spectra configuration of PETLS with HERMES and the MATLAB implementation on a logarithmic scale. This demonstrates that this usage of PETLS dramatically outperforms both HERMES and MATLAB at each filtration step. 

It is important to note that the usage of the method from Spectra requires choosing the number of eigenvalues and whether to select the largest or smallest. Because of these choices, and the fact that it fails to converge in common situations, it is not fully generalizable. However, it extracts the maximum performance out of all of the methods. With the exception of this Spectra method, the dominant problem has now changed: in HERMES and the MATLAB implementation, the bottleneck is computing the matrix $\Delta_n^{a,b}$, while in PETLS the bottleneck is typically the eigenvalue computation. This is good news, as there are many ways to approach this problem. Implicitly restarted Lancoz/Arnoldi methods from Spectra are one way, but more sophisticated approaches are possible. 

In Table \ref{tab:compare_sphere_rips} we report the combined computation time for both the matrix and eigenvalues, across all dimensions and filtration steps for each software, averaged over the $100$ random samples. The main observation is that even the slowest PETLS configuration is at least $8$ times faster than the MATLAB implementation. With further optimizations, we are able to reduce a $22$ minute HERMES computation or a $1$ minute MATLAB computation to under $5$ seconds, representing a speedup by factors of over $300$ and $15$, respectively. Moreover, this ignores the heavy practical computational cost of starting MATLAB that occurs with each execution of HERMES. The speedup from the slowest PETLS configuration to the fastest is by a factor of about $2$.

\begin{table}
	\centering
	\begin{tabular}{lr|lll}
		\toprule
		Software & Time (s) & \multicolumn{3}{c}{Speedup Over} \\
		 &  & HERMES \cite{wang2021hermes} & MATLAB \cite{memoli_PL} & \makecell{PETLS\\ (C++, Spectra)} \\
		\midrule
			HERMES & 1359.303 & - & - & - \\
			MATLAB & 68.289 & 19.91 & - & - \\
			PETLS (C++, Spectra) & 8.268 & 164.41 & 8.26 & - \\
			PETLS (LAPACK) & 7.349 & 184.96 & 9.29 & 1.12 \\
			PETLS (Python, Spectra) & 5.231 & 259.86 & 13.05 & 1.58 \\
			PETLS (Python) & 5.074 & 267.90 & 13.46 & 1.63 \\
			PETLS (C++) & 4.332 & 313.78 & 15.76 & 1.91 \\
		\bottomrule
	\end{tabular}
	\caption{Absolute execution time and relative speedup using different software to compute all persistent Laplacian matrices $L_d^{a,a+0.2}$ and their eigenvalues for Rips complexes built from $30$ points randomly sampled from the unit sphere, where $d\in\{0,1,2\}$ and $a\in\{0.0,0.2,0.4,\dots,2.0\}$. Time is averaged over $100$ replications. Relative speedup is the time taken by the software in the column divided by the time taken by the software in the row. Each row reports the performance gains of that software.}
	\label{tab:compare_sphere_rips}
\end{table}

\subsection{Eigenvalue algorithms}

Since the eigenvalue computation time is critical, yet possible to manage, we further interrogate why this takes so long and determine which eigenvalue algorithms are appropriate. We compute the eigenvalues of several persistent Laplacian matrices to show that persistent Laplacian matrices are a particularly challenging eigenvalue problem and some ways to mitigate it.

We compute the alpha complex of $100$ randomly sampled tori, each with an inner radius of $3$, an outer radius of $1$, and $500$ points, generated using Scikit-TDA's TaDAsets \cite{scikittda2019}. We compute the persistent Laplacian matrices $\Delta_i^{a,a+\delta}$ for $i\in\{0,1,2\}, a\in\{0.0,0.1,\dots, 5.0\}$ and $\delta=0.1$ in dimensions $i=0,1,$ and $2$. Then we compute the eigenvalues of these matrices using several algorithms. The full spectrum of eigenvalues for each matrix is computed using Eigen's generic \texttt{EigenSolver}, Eigen's self-adjoint specialized algorithm, and SciPy's \texttt{eigvalsh} algorithm, which uses LAPACK methods for symmetric matrices. We compute the largest and smallest $10$ eigenvalues using Spectra and SciPy's sparse eigenvalue solvers. For the smallest eigenvalues, we use the shift-and-invert method, which is more stable and efficient than other methods. When the Spectra method fails to converge, Eigen's self-adjoint method is used, and when SciPy's smallest eigenvalues method fails to converge, the full-spectrum SciPy method is used.

Figure \ref{fig:eigenvalue_algorithms_torus_d2} shows how the computation time scales with respect to the number of simplices (equivalently, rows in the matrix) in dimension $2$. The rank ordering is fairly stable above $1000$ simplices, regardless of matrix size. The algorithms that compute a subset of the eigenvalues generally perform faster at all scales, save for Spectra's smallest-eigenvalues. This is largely because the algorithm frequently fails to converge for these persistent Laplacian matrices. Also observe that SciPy's smallest and largest eigenvalue methods generally take the same amount of computation time for persistent Laplacians with up to $4000$ simplices. In general, computing the smallest eigenvalues is slower and less robust than computing the largest, but it may not be meaningfully worse in the context of persistent Laplacians of this size.

Table \ref{tab:eigenvalue_algorithms} shows the average time (across the $100$ tori) to compute all of the eigenvalues in dimensions $i=0,1,$ and $2$. Observe that the generic eigenvalue algorithm from Eigen consumed over $13,000$ seconds of computation time ($3$ hours and $51$ minutes). If not managed correctly, the choice of eigenvalue algorithm can quickly make computing persistent Laplacian eigenvalues intractable. Eigen's self-adjoint method is the natural appropriate method, but far from optimal. Spectra's methods are a strict improvement over Eigen's methods, but not nearly as much as SciPy's smallest or largest eigenvalue algorithms. SciPy's full-spectrum method offers a modest improvement over Eigen's self-adjoint. If the smallest or largest few eigenvalues convey sufficient information about a complex for an application, then we may need as little as $6$ seconds - beating the generic approach by a factor of over $2000$ and the reasonable alternative (Eigen's self-adjoint method) by a factor of over $300$. 

\begin{figure}[htpb]
    \centering
    \begin{subfigure}{5.4in}
       \centering
        \includegraphics[width=5.4in]{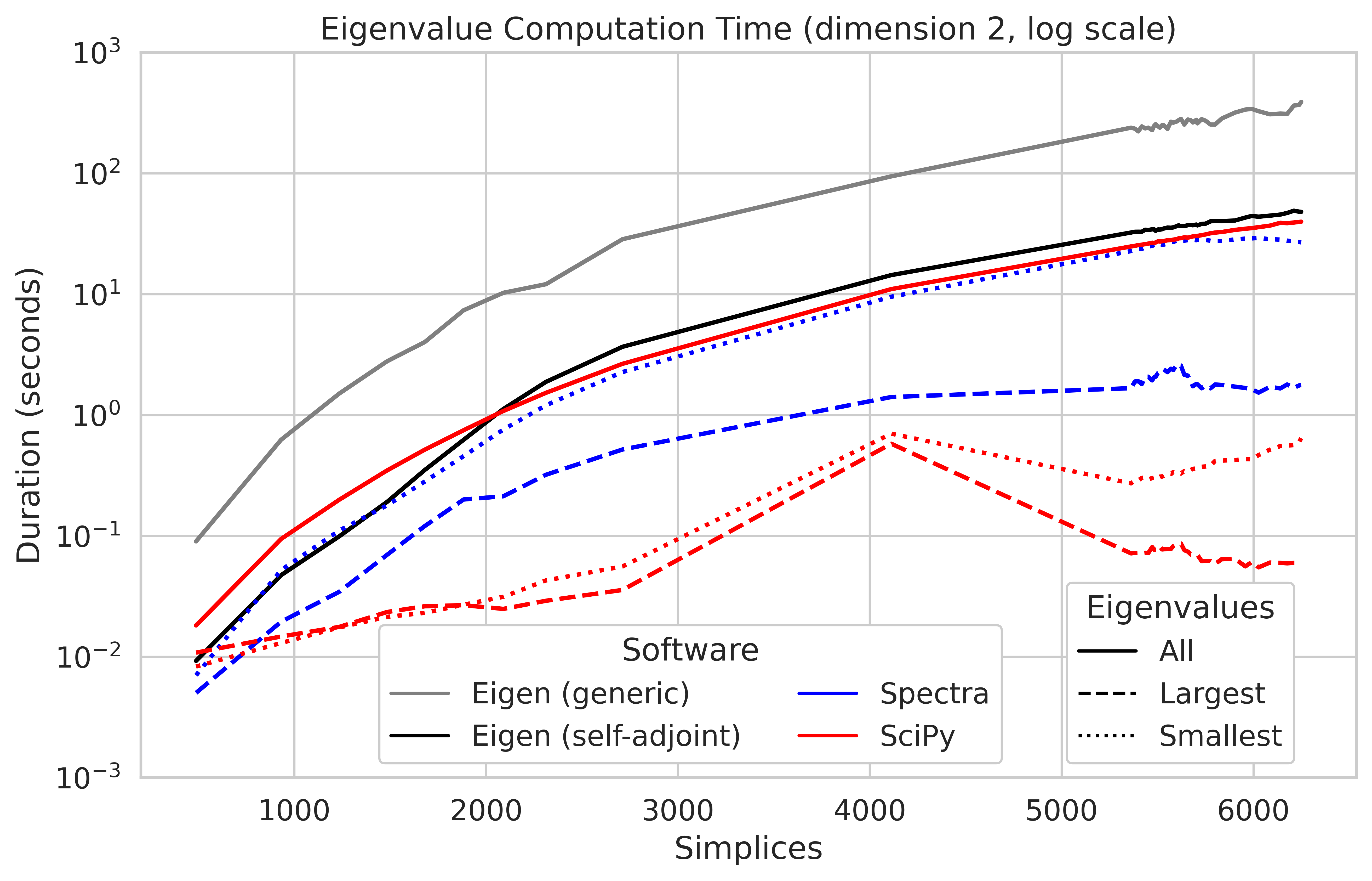}
    \end{subfigure}
\caption{Time to compute the eigenvalues of pre-computed persistent Laplacian matrices of Alpha complexes on Tori using different eigenvalue computation software. Results are averaged over $100$ replicates.}
\label{fig:eigenvalue_algorithms_torus_d2}
\end{figure}

\begin{table}
\centering
\begin{tabular}{lr|rr}
\toprule
 & & \multicolumn{2}{c}{Speedup Over} \\
 & Time(s) & Eigen (generic) & Eigen (self-adjoint) \\
\midrule
Eigen (generic) & 13891.9 & - & - \\
Eigen (self-adjoint) & 1943.2 & 7.1 & - \\
SciPy & 1538.9 & 9.0 & 1.3 \\
Spectra (smallest) & 1343.3 & 10.3 & 1.4 \\
Spectra (largest) & 110.9 & 125.2 & 17.5 \\
SciPy (smallest) & 33.2 & 418.7 & 58.6 \\
SciPy (largest) & 6.3 & 2200.4 & 307.8 \\
\bottomrule
\end{tabular}
\caption{Execution time and relative speedup using different eigenvalue computation software for persistent Laplacians $\Delta_i^{a,a+\delta}$ of tori in dimensions $0, 1, $ and $2$ and $a\in\{0.0,0.1,\dots, 5.0\}$ with $\delta=0.1$.}
\label{tab:eigenvalue_algorithms}
\end{table}

The algorithms that compute a subset of eigenvalues are generally faster, and in some sense this is the expected outcome. It is not always the case with persistent Laplacians, however. If we compute the eigenvalues of persistent Laplacians where the complex has homology with high rank, this gives a high multiplicity of zero as an eigenvalue of the persistent Laplacian. This destabilizes the algorithms that compute a subset of eigenvalues.

To see this experimentally, we consider an example from \cite{jones_persistent_2025}. A directed graph is built using 3D coordinates of atoms in a protein-ligand complex as vertices, and we compute the persistent Laplacian of the directed flag complex of this directed graph. Edges are added between the vertices corresponding to a protein atom $P$ and a ligand atom $Q$ with a filtration value $r$ equal to the Euclidean distance between $P$ and $Q$. Edges are added between the vertices corresponding to two ligand atoms $Q_1$ and $Q_2$ also at filtration value equal to the Euclidean distance, and only if there is a chemical bond between $Q_1$ and $Q_2$. No edges are added for vertices corresponding to two protein atoms. The directions of these edges is determined by the element's electronegativity, as a model of the direction electric charge might flow. When two atoms $A_1$ and $A_2$ are of the same element and would have an edge added at filtration $r$, the edges $A_1\to A_2$ and $A_2\to A_1$ are both added at filtration $r$, since their elemental electronegativity is equal. Finally, only protein atoms within a threshold distance of a ligand atom are included. This gives a maximum filtration value to consider. This process produces a filtered directed graph $G$, and from this we can build the filtered directed flag complex $dFl(G)$. We do this for an example protein-ligand complex, PDBID  1A99 \cite{1a99}. Then, just as with the tori, we compute a collection of persistent Laplacian matrices and run those through the battery of eigenvalue algorithms. We choose $\Delta_i^{a,a+\delta}$ for $a\in\{0.0,0.1,\dots,10.0\}, \delta=0.1,$ and $i\in\{1,2\}$.

The resulting behavior of some eigenvalue methods does not comply with the standard guideline. SciPy's sparse Hermitian eigenvalue algorithm \texttt{eigsh} uses the implicitly restarted Lanczos algorithm from ARPACK. This algorithm is iterative and depends on the choice of an error tolerance $\texttt{tol}$ and the number of vectors used in the Krylov subspace $\texttt{ncv}$. To compute $k$ eigenvalues and eigenvectors of an $n\times n$ hermitian matrix, it is necessary that $k<n$ and $k<\texttt{ncv}<n$, and it is often recommended to use $\texttt{ncv}\geq 2k$ \cite{arpack_user_guide}. In general, a lower tolerance increases the number of iterations, and therefore computation time, but provides more stable and accurate results. Increasing the number of Lanczos vectors increases the cost of each iteration and increases the memory usage.  Parameters $\texttt{tol}$ and $\texttt{ncv}$ should be chosen in a way that reliably yields accurate results, but without extreme computational demands. 

In Figure \ref{fig:ncv_tol_heatmap} we show the time to compute $k=10$ eigenvectors for a variety of choices of $\texttt{tol}$, from machine precision $\epsilon$ to $10^{-4}$, and a variety of Lanczos vectors $\texttt{ncv}$ from $20=2\cdot k$ to $250=25\cdot k$. Results where the algorithm failed to converge within $1000$ iterations or converged to an incorrect eigenvalue are omitted. The best choice of parameters should be where computation time is the lowest with acceptable confidence in the numerical quality of the results.

The generally recommended $\texttt{ncv}=2k$ did not converge for any sampled tolerance, including the default of machine precision $\epsilon$, and the algorithm did not reliably converge until $\texttt{ncv}\geq 60=6k$, much higher than the standard recommendation. As $\texttt{ncv}$ increased toward $250$, the computation time settled for all tolerances to a range of $35-46$ seconds. Computation times of under $30$ seconds were generally achieved for tolerances $10^{-4}\leq \texttt{tol}\leq 10^{-11}$ and number of Lanczos vectors $6k=60\leq \texttt{ncv}\leq 100=10k$.  Importantly, the default tolerance of machine precision $\epsilon\approx 10^{-16}$ was significantly slower than other methods and often failed to converge.
{\small 
\begin{figure}
    \centering
    \includegraphics[width=6in]{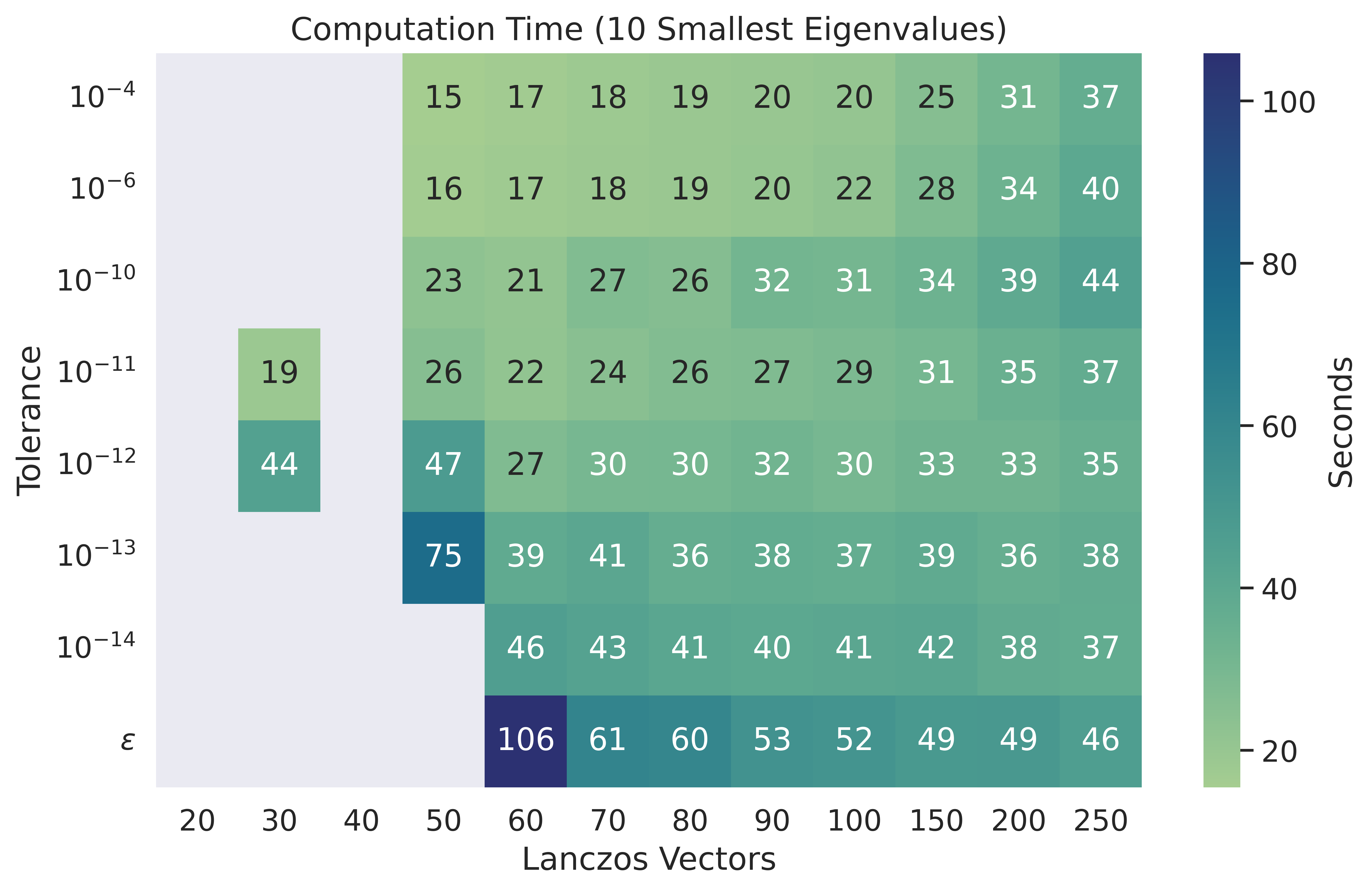}
    \caption{Eigenvalue convergence and computation time vary when using the implicitly restarted Arnoldi method of SciPy. This figure shows how these scale on a directed flag complex associated to a directed graph built on a protein-ligand complex (PDBID: 1A99), with respect to two key hyperparameters: the convergence tolerance and the number of Lanczos vectors in the Krylov subspace. Standard guidance suggests using parameter combinations that are not suitable for persistent Laplacians.}
    \label{fig:ncv_tol_heatmap}
\end{figure}
}

Now that a reliable range of $\texttt{ncv}$ and $\texttt{tol}$ has been established for this problem with the SciPy smallest-eigenvalue algorithm, we can compare some of these against other eigenvalue methods. In Table \ref{tab:protein_dflag} we show the execution time for several methods, where we compute the same spectrum of the directed flag complex, averaged over $100$ executions. For SciPy's smallest-eigenvalue algorithm, we took $\texttt{ncv}=60$, while for Spectra's we took $\texttt{ncv}=20$, since this was sufficient for reliable numerical results for this complex. Although SciPy uses the actual ARPACK routines, Spectra uses a re-implementation of the algorithms. Spectra methods use a different convergence criterion to stop iterations, which can result in more stable performance (here) or less stable performance, as we saw in Table \ref{tab:eigenvalue_algorithms}.

SciPy's smallest eigenvalue algorithm with $\texttt{tol}=\epsilon$ took $6$ times longer than the Eigen self-adjoint method. In fact, the only iterative method that performed substantially better than the full-spectrum methods was Spectra's largest-eigenvalue computation. This is a categorically different ranking from that observed in the torus experiment, where SciPy's smallest and largest eigenvalue algorithms were much faster than all other methods. 

\begin{table}
\centering
\begin{tabular}{lr|rrr}
\toprule
 & \multicolumn{3}{c}{Speedup Over} \\
 & Time & \makecell{Eigen\\ (generic)} & \makecell{SciPy\\ (smallest, $\texttt{tol}=\epsilon$)} & \makecell{Eigen\\ (self-adjoint)} \\
\midrule
Eigen (generic) & 388.1 & 1.0 & 0.4 & 0.1 \\
SciPy (smallest, $\texttt{tol}=\epsilon$) & 151.4 & 2.6 & 1.0 & 0.2 \\
Spectra (smallest) & 45.6 & 8.5 & 3.3 & 0.5 \\
Eigen (self-adjoint) & 23.9 & 16.2 & 6.3 & 1.0 \\
SciPy & 22.3 & 17.4 & 6.8 & 1.1 \\
SciPy (smallest, $\texttt{tol}=10^{-6}$) & 18.7 & 20.8 & 8.1 & 1.3 \\
SciPy (largest) & 17.3 & 22.5 & 8.8 & 1.4 \\
Spectra (largest) & 2.1 & 189.0 & 73.7 & 11.6 \\
\bottomrule
\end{tabular}
\caption{Computation times for different eigenvalue algorithms for the directed flag complex of a protein-ligand complex.}
\label{tab:protein_dflag}
\end{table}

Together, these experiments suggest a method for choosing eigenvalue algorithms. If the goal is to find the largest eigenvalues, it is likely there are \textit{some} iterative methods that are much faster than other iterative methods or full-spectrum methods. If the goal is to find the smallest eigenvalues, it may be more effective to use full-spectrum methods, but an iterative method may still be the best. These trends tend to reveal fairly early in a filtration, as in Figure \ref{fig:eigenvalue_algorithms_torus_d2}, where the rank ordering remained stable above $2,000$ simplices, which all algorithms (other than Eigen's generic) could compute in $1$ second or less. Therefore, a trial-and-error approach on a subset of data may be a reasonable method for scaling up in a given problem.

We have shown the persistent Laplacian eigenvalue problem does not obey standard rules of thumb, and this can be partially attributed to the underlying topology, since the dimension of homology corresponds to the presence of $0$ as an eigenvalue.  High-rank persistent homology is expected for short-lived features, which appear as points near the diagonal on a persistence diagram or as short bars on a barcode. The optimal eigenvalue algorithm can therefore depend on which persistent Laplacians are being computed for a complex, where sparse methods will generally be the most effective for persistent Laplacians that span larger filtration intervals. In some settings, a good choice of hyperparameters may differ substantially from general recommendations and defaults. Rather than optimizing a choice of parameters for a particular eigenvalue algorithm in a problem-dependent fashion, we can mitigate this problem by more sophisticated methods, which are discussed in Sections \ref{subsec:flipped} and \ref{subsec:reduction}. 

Another possible approach would be to use homotopy continuation as in \cite{wei2021homotopy}. However, current homotopy continuation methods are not computationally scalable for large data.

\subsection{Top-dimensional Laplacian eigenvalues}\label{subsec:flipped}
The highest dimensional Laplacian $\Delta_N^{a,b}=\Delta_{N,\mathrm{down}}^{a}$ is sometimes the slowest to compute the eigenvalues for, but a small modification can greatly reduce this bottleneck. The following standard linear algebra fact allows us to compute the spectra of these Laplacians twice as fast.

\begin{theorem}
    Suppose $A\in\mathbb{R}^{m\times n}, B\in \mathbb{R}^{n\times m}.$ If $\lambda \neq 0$ is an eigenvalue of $AB\in\mathbb{R}^{m\times m}$, then $\lambda$ is an eigenvalue of $BA\in\mathbb{R}^{n\times n}$.
\end{theorem}
\begin{proof}
    Suppose $x$ is the eigenvector of $AB$ associated with $\lambda$, so that $ABx=\lambda x$. Then $BABx=B\lambda x$, or $BA(Bx)=\lambda(Bx)$, so that $Bx$ is an eigenvector of $BA$ with eigenvalue $\lambda.$
\end{proof}

This means that the nonzero eigenvalues of $\Delta_{N,\mathrm{down}}^{a}=B_N^a(B_N^a)^T$ and $\Delta_{N-1,\mathrm{up}}^{a}=(B_N^a)^TB_N^a$ are the same. This is noted in \cite{goldberg2002combinatorial}, but has not been implemented in any persistent Laplacian software. We can therefore choose to compute the nonzero eigenvalues of the smaller matrix and $0$-pad to obtain the full set of eigenvalues.

Consider the directed flag complex discussed in Section \ref{subsec:computation}. This produces a directed graph with $209$ vertices, $1,466$ edges, and $2,361$ $3$-cliques (2-simplices in the directed flag complex), with filtration values between $1.5${\AA} and the cutoff value of $10${\AA}. We then compute the persistent Laplacians $\Delta_d^{a,a+0.1}$ for $d\in\{1,2\}$ and $a\in\{0.0, 0.1,0.2,\dots,10.0\}$ using the two different methods to compute the top-dimensional Laplacian (in this case, $N=2$). This is repeated $100$ times. Figure \ref{fig:flipped}(a) shows the average cumulative computation time of the eigenvalue steps, where we can see that it takes the same amount of time to compute the eigenvalues of $\Delta_{N,\mathrm{down}}^{a}$ via $\Delta_{N-1,\mathrm{up}}^{a}$ as it does to compute the eigenvalues of $\Delta_{N-1}^{a,b}$.  The dimension $2$ eigenvalue computation moves from $13.1$s to $3.5$s, a $73\%$ reduction in computation time. Figure \ref{fig:flipped}(b) quantifies the impact on the overall computation: this method amounts to a greater speedup than $2$-fold in the time it takes to compute the full collection of persistent Laplacian matrices and their eigenvalues in dimensions $1$ and $2$.

\begin{figure}
\centering
    \begin{subfigure}{3.5in}
       \centering
        \includegraphics[width=3.5in]{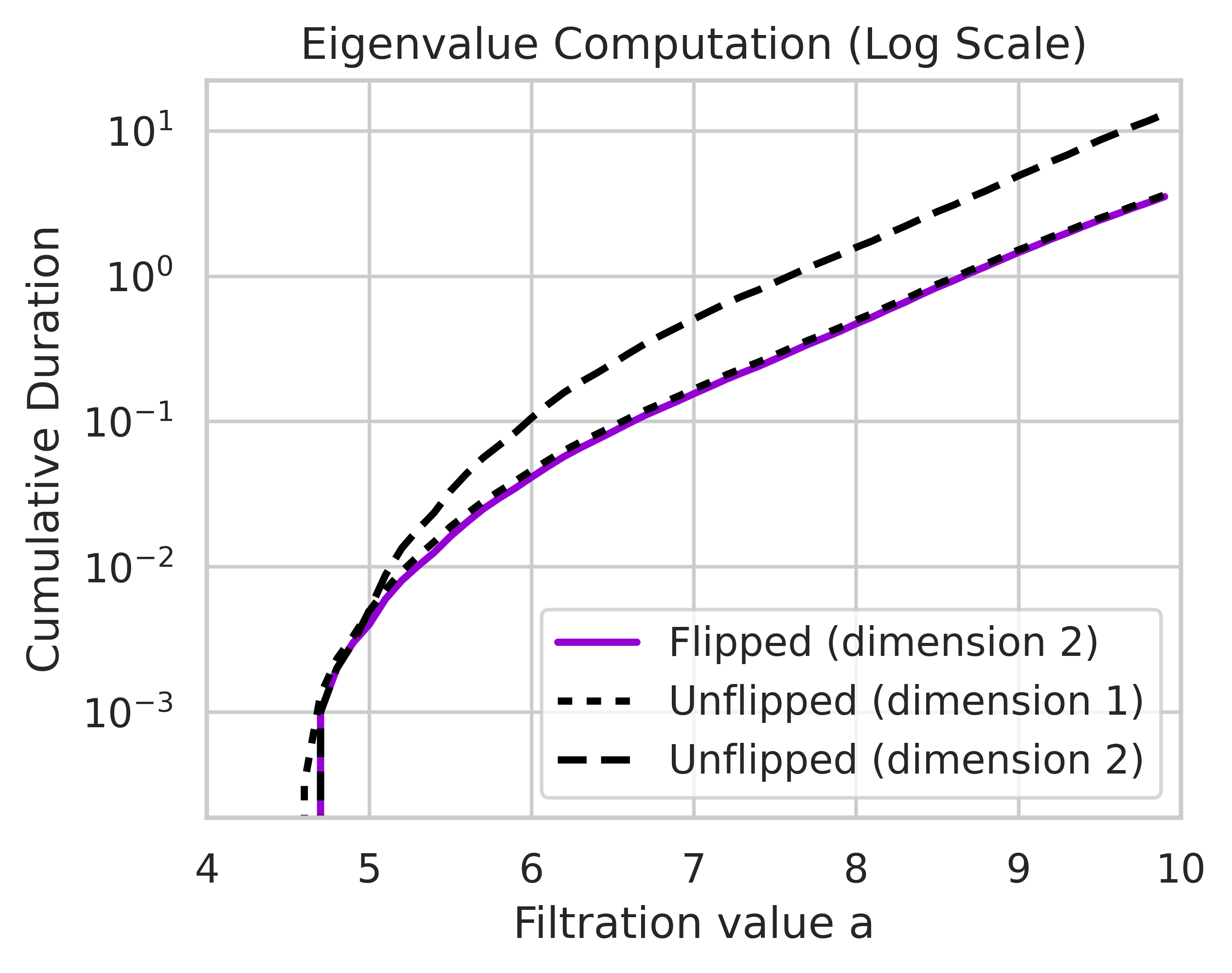}
        \caption{}
\end{subfigure}\hfill\begin{subfigure}{3.5in}
       \centering
        \includegraphics[width=3.5in]{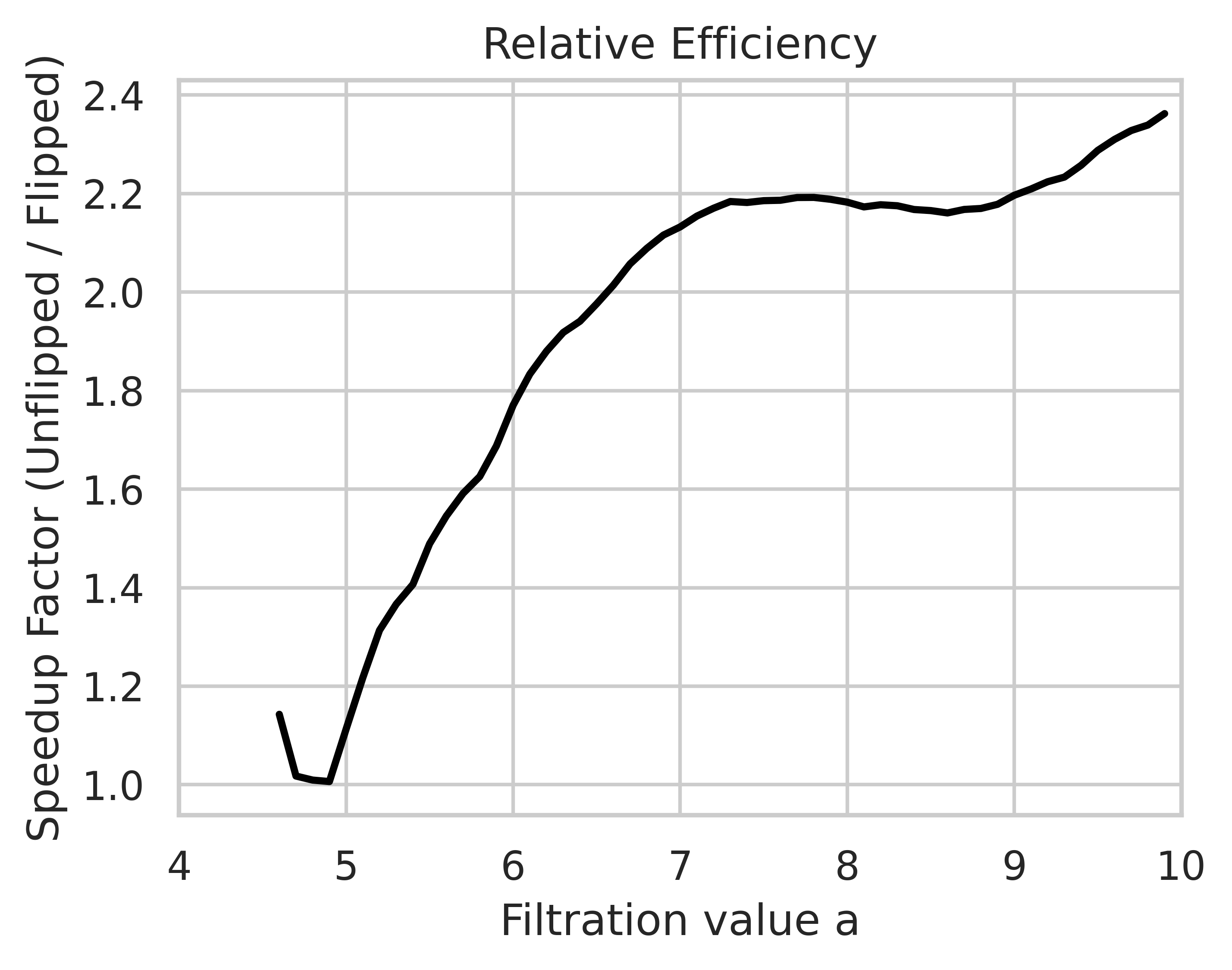}
        \caption{}
    \end{subfigure}
    \caption{Efficiency comparison for computing the top-dimensional Laplacian eigenvalues of a directed flag complex built on a protein-ligand complex using the traditional method and the ``flipped" method. (a) Cumulative time with respect to the filtration to compute the eigenvalues of the persistent Laplacian matrices. (b) The relative speedup in cumulative computation time for both matrix and eigenvalue computations in combined dimensions $1$ and $2$.}
    \label{fig:flipped}
\end{figure}

The benefit occurs when $B_N$ has many more columns than rows, or when the complex has many $N$ cells and not many $N-1$-cells. In the most extreme case, the Vietoris-Rips complex truncated at dimension $N$ with no radius threshold on a point cloud of size $k$ has $n_N =\binom{k}{N}$ many $N$-cells and $n_{N-1}=\binom{k}{N-1}$ $N-1$-cells. Then there are $n_N-n_{N-1}=\frac{k!(k-2N+1)}{N!(k-N+1)!}$ more $N$-cells than $N-1$-cells. For a fixed dimension $N$, the difference in matrix size is $O(k^N)$. This is essentially the most dramatic case, and in practice much smaller complexes are used, but this shows that there are complexes where this method can offer arbitrarily large benefits.

\subsection{Reduction by persistent homology}\label{subsec:reduction}
One of the central themes of persistent Laplacian is that it can be used to compute persistent Betti numbers, something that is well-studied from the perspective of persistent homology. An active area of research is the computation of topological invariants of persistent homology as in \cite{basu2024harmonic}, which can be formulated through the harmonic spectra given by the kernel of a Laplacian. We have two ways to study the same mathematical object, and so the two perspectives may inform each other. In fact, we can use persistent homology to isolate the information unique to persistent Laplacians. This leads to an algorithm that reduces the size of the current practical bottleneck, computing the eigenvalues. We use Schur restriction from \cite{gulen_et_al:LIPIcs.SoCG.2023.37} to restrict the Laplacian to the complement of the harmonic subspace. 

Suppose $d=\dim C_n^a$ and $L_n^{a,b}\in\mathbb{R}^{d\times d}$ is a persistent Laplacian. Let $\{\sigma_1,\dots,\sigma_{\beta_n^{a,b}}\}$ be a basis of cycle representatives for $H_n^{a,b}(K;\mathbb{R})$. If each $\sigma_i$ is written as a vector $\sigma_i=[\alpha^i_0,\dots,\alpha^i_{d}]^T$, we can form the matrix $N=\begin{bmatrix}\sigma_1 & \dots & \sigma_{\beta_n^{a,b}}\end{bmatrix}\in\mathbb{R}^{d\times \beta_n^{a,b}}$ where each column is the vector corresponding to a representative cycle. By the isomorphism $\ker\Delta_n^{a,b}\cong H_n^{a,b}$, we have that $N$ is a basis for $\ker\Delta_n^{a,b}$. Let $X\in\mathbb{R}^{d\times(d-\beta_n^{a,b})}$ be a basis for the orthogonal complement of $\ker\Delta_n^{a,b}$. Then the block matrix $S=\begin{bmatrix} X & N \end{bmatrix}\in\mathbb{R}^{d\times d}$ can be used as a change of basis for $L$:

\begin{align*}
    \tilde{L} &= S^{-1} L S\\
    &= S^{-1} L \begin{bmatrix} X & N \end{bmatrix}\\
    &= S^{-1}\begin{bmatrix} LX & LN\end{bmatrix}\\
    &= S^{-1}\begin{bmatrix} LX & 0\end{bmatrix}\\
    &= \begin{bmatrix} S^{-1} LX & 0\end{bmatrix}.
\end{align*}

Note that $S^{-1}LX\in\mathbb{R}^{d\times(d-\beta_n^{a,b})}$ and $0=LN\in \mathbb{R}^{d\times \beta_n^{a,b}}$. Note that $\tilde{L}$ has the same eigenvalues as $L$, while satisfying the requirements of the Schur restriction: $\tilde{L}=\begin{bmatrix}A & B\\C & D\end{bmatrix}$, where $A\in\mathbb{R}^{(d-\beta_n^{a,b})\times(d-\beta_n^{a,b})}$ and $D\in \mathbb{R}^{\beta_n^{a,b}\times\beta_n^{a,b}}$. Since $LN=0$, we have $C=0$ and $D=0$. Then the Schur complement $\tilde{L}/D=A-BD^{-1}C=A$ is the restriction of $L$ onto $\left(\ker\Delta_n^{a,b}\right)^\perp$. This is the top $(d-\beta_n^{a,b})$ rows of $\begin{bmatrix}X & N\end{bmatrix}^{-1}LX$. Thus, this represents the restriction of persistent Laplacian to the non-harmonic space, and we have effectively removed the information that persistent Laplacian contains about persistent homology.

Since $B=0$, $\tilde{L}$ is a block lower-triangular matrix and its eigenvalues are the combination of the eigenvalues of $A$ and $D$. Since $D=0$, and we have a correspondence between the rows in $D$ and $0$-eigenvalues of $\tilde{L}$, we know that the eigenvalues of $A$ are all of the nonzero eigenvalues of $\tilde{L},$ and more importantly of $L$. Therefore, we can compute the eigenvalues of the positive definite (no longer positive semi-definite) matrix $A\in\mathbb{R}^{(d-\beta_n^{a,b})\times(d-\beta_n^{a,b})}$. Because there is no high multiplicity of $0$ as an eigenvalue, eigenvalue algorithms should be faster and more stable on $A$. The matrix $A$ is also a smaller matrix than $L$, further contributing to the benefit.

The expected benefits of this method are difficult to quantify, as the motivation is due to the failure of the randomized iterative algorithms on the Eigenvalue problem, which depends in a non-precise way on the size of the null space. However,  the costs are easier to quantify. First, we must obtain persistent homology with representatives and real coefficients. This could be done once for the filtered complex and the results could be used for all of the possible persistent Laplacians. This one-time cost is negligible compared to the recurring cost of the change in basis. We must compute the top $(d-\beta_n^{a,b})$ rows of $\begin{bmatrix}X & N\end{bmatrix}^{-1}LX$. Directly computing $S^{-1} L S$ requires one $d\times d$ matrix inverse and two $d\times d$ matrix multiplications, which gives it the same complexity as matrix multiplication. In practice, this is approximately $O(d^3)$, but can be reduced either by using optimized matrix multiplication algorithms or using a linear solver  for $\begin{bmatrix}X & N\end{bmatrix}^{-1} L$, rather than computing an actual inverse.

Computing persistent homology is generally much faster than persistent Laplacians, but there are no major Python or C++ libraries that compute persistent homology with representatives over $\mathbb{R}$ (or $\mathbb{Q}$), and the representatives of homology classes with coefficients in $\mathbb{Z}_p$ will not be the same. The only major existing software package that supports $\mathbb{Q}$ coefficients is JavaPlex, which would introduce complex dependencies and communication between programs. We therefore decouple our software from the computation of persistent homology and take as input a basis of the null space for this algorithm. This allows for persistent homology of the real coefficient to be done externally. By using a pre-computed null space, we can estimate potential performance gains.
From our first approach at this, we find that it offers only a modest speed benefit. However, this is a new algorithm for approaching the relationship between persistent topological Laplacians and persistent homology, and we therefore hope that it may be optimized both in theory and in practice. One possible avenue for this is to reduce the boundary operator before computing $\Delta_n^{a,b}$. 

\section{From large complexes to meaningful spectra}
\label{sec:representations}
We have found several techniques, algorithms, and subtleties in efficiently computing persistent Laplacians. We now want to know how large of a complex we can analyze - and what it means to ``analyze" a complex. The first part is straightforward: compute some selection of persistent Laplacians and their eigenvalues for increasingly large complexes.

To ``analyze" a filtered complex with persistent homology, one often computes the persistence diagram (or the equivalent barcode), and possibly a vectorization. The persistence diagram is the unique decomposition of a persistence module as a direct sum of interval modules. This does not involve the choice of which filtration intervals $(a,b)$ to consider. Instead, we obtain a homological characterization of all of the relevant filtration values.

In contrast, we compute a collection of persistent Laplacians $\Delta_n^{a,b}$ for different filtrations $a$ and $b$. If the set of distinct filtration values is $A=\{a_0,a_1,\dots, a_k\}$, then we can compute all of the possible persistent Laplacians $\{\Delta_n^{a_i,a_j} | a_i, a_j \in A, i\leq j\}$. As we have seen, computing some persistent Laplacian matrices and eigenvalues can be computationally demanding. The number of pairs $a_i,a_j\in A$ with $i\leq j$ is $\frac{|A|(|A|+1)}{2},$ and $|A|$ can be as large as the number of simplices in the complex, if each has a unique filtration value. In Table \ref{tab:eigenvalue_algorithms} we reporte times used in computing the eigenvalues of a complex with $500$ $0$-simplices, $3,850$ $1$-simplices, and $6,249$ $2$-simplices, for a total of $10,599$ simplices. The fastest eigenvalue computation took $6.3$ seconds. If each simplex has a unique filtration value, there are possibly more than $5.6$ million distinct persistent Laplacians. Proportionally, the eigenvalue computations would take over $27$ days. Although Theorem 5.2 of \cite{memoli_PL} describes an algorithm for efficiently computing all pairs of persistent Laplacian matrices, the eigenvalue computation time alone requires us to compute a strict subset of the possible persistent Laplacians. We must make a choice of what that subset is.

One possible method is to consider $\Delta_n^{b_i,b_i+\delta}$ for some fixed $\delta$ and each $b_i\in B\subset A$ \cite{wang2020PSG}, which we have done in some of the examples in this work. This representation is shown in Figure \ref{fig:curve_sphere_rips} for Rips complexes built on points sampled from a sphere. Another method is to consider $\Delta_n^{a_i,a_{i+1}}$ for each filtration value $a_i\in A$, which models the changes between each step \cite{jones_persistent_2025}. A third straightforward approach is to compute a sequence of persistent Laplacians with no persistence, $\Delta_n^{a_i,a_i}$. Each of these produces a $1$-dimensional family of descriptors. One strategy for reducing adaptive computation is to select which persistent Laplacians to compute first by computing persistent homology, then computing only those $\Delta_n^{a,b}$ for which $\beta_n^{a,b}$ is small or zero. In addition to requiring some choice on part of the researcher, these selections miss a core feature of persistent homology: that features change over a variety of scales.

We can instead consider a $2$-dimensional family of features by taking a subset $B=\{b_0,b_1,\dots, b_\ell\}\subset A$ and computing all pairs $\Delta_n^{b_i,b_j}$ for $i\leq j$. Moreover, we can take $(i,j)$ or $(b_i,b_j)$ as coordinates, and at each lattice point display the full distribution of eigenvalues, or some feature we compute from the set of eigenvalues, such as the minimum nonzero eigenvalue. 

In Figure \ref{fig:all_pairs_torus}(a) we again sample $500$ points from a torus, compute the eigenvalues of $\Delta_1^{a,b}$ for an alpha complex, and display histograms of the nonzero eigenvalues. Here we do all possible pairs $(a,b)\in B=\{0.0,1.0,2.0,3.0,4.0\}$ with $a\leq b$ and plot the diagonal as is done with persistence diagrams. Note that there are no $1$-cells at filtration $0$, so there are no nonzero eigenvalues to display for $\Delta_1^{0,b}$. In Figure \ref{fig:all_pairs_torus}(b) we compute a feature (the minimum nonzero eigenvalue) for each of these distributions, which is a triangular matrix that can be displayed as a heatmap or image. The resulting image can then be used for analysis in a variety of ways, such as producing vectorizations for machine learning by flattening or sub-sampling entries from the matrix, transforming the image in a way that models the persistence image \cite{adams_persistence_image}, using a convolution neural network on the image layer\cite{chen2025drug}, or visualization and qualitative analysis.

\begin{figure}[htpb]
    \centering
    \begin{subfigure}{3.0in} 
       \centering
        \includegraphics[width=2.6in,height=2.6in]{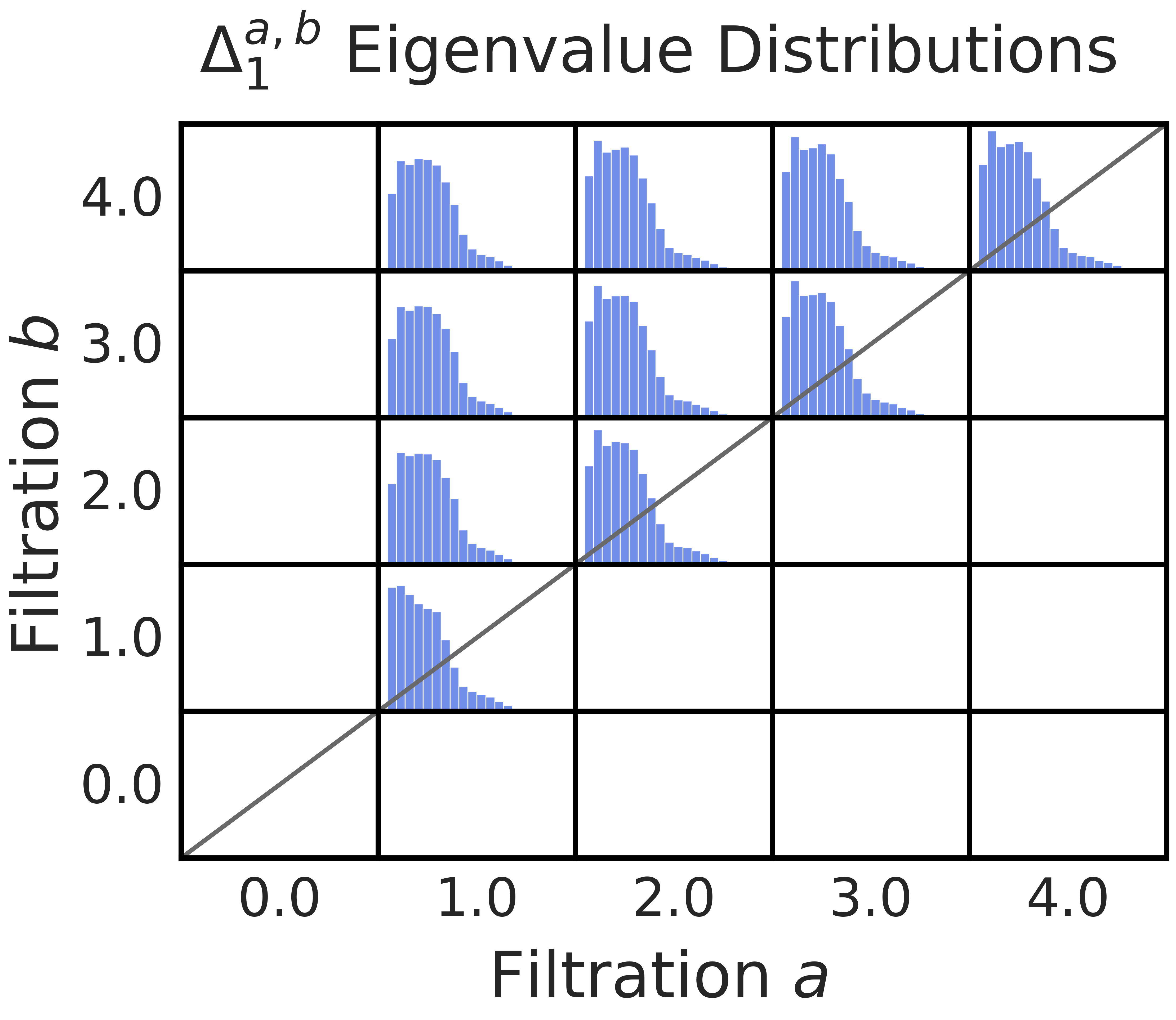}
        \caption{}
    \end{subfigure}
    \hfill
    \begin{subfigure}{3.0in}
        \centering
        \includegraphics[width=2.8in,height=2.8in]{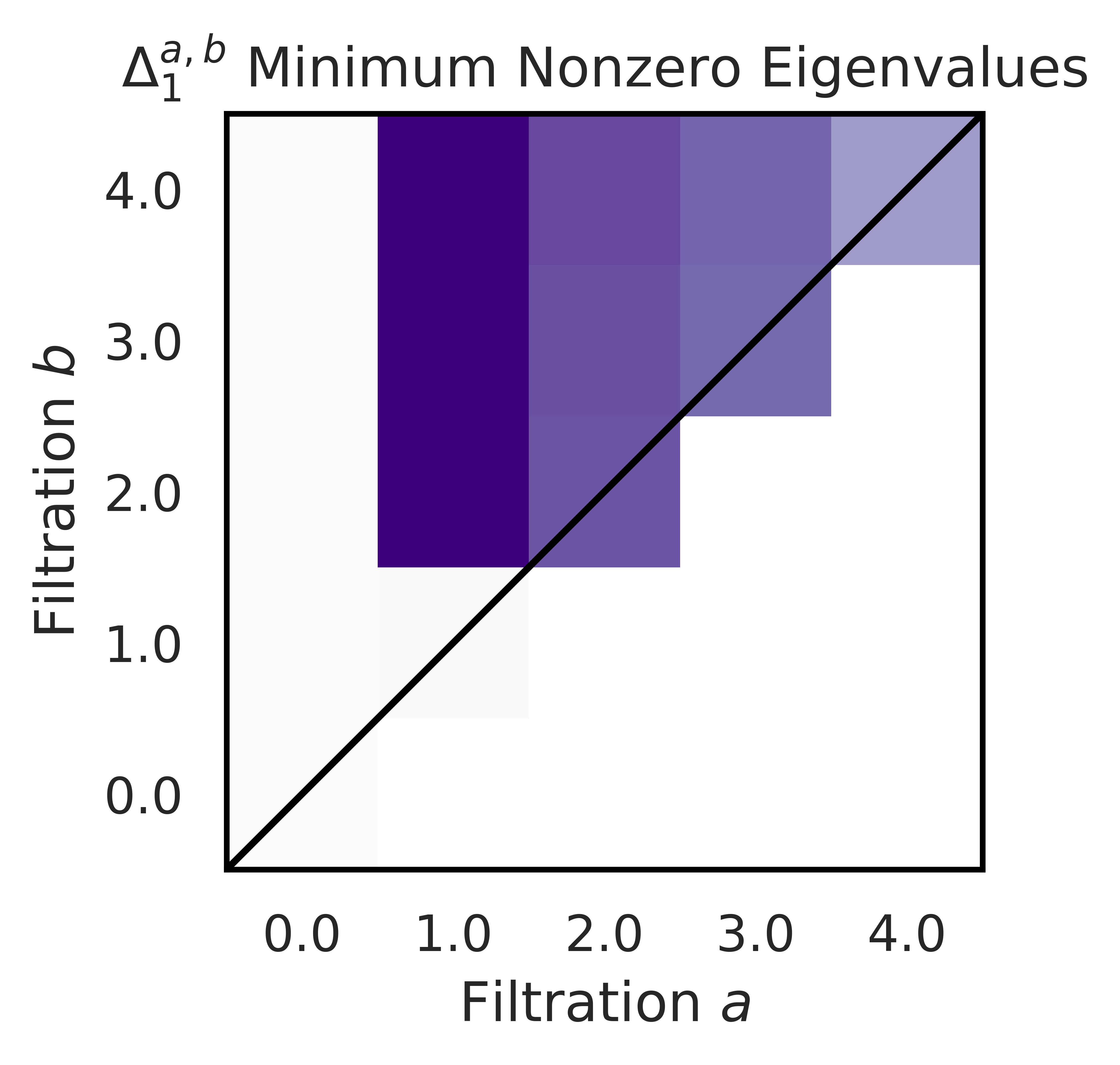}
        \caption{}
    \end{subfigure}
\caption{A $2$-dimensional family of persistent Laplacians $\Delta_1^{a,b}$ is computed for an Alpha complex built on a torus with $500$ points. All pairs with $a,b\in\{0,1,2,3,4\}$ and $a\leq b$ are computed. (a) Nonzero eigenvalue distributions of $\Delta_1^{a,b}$. (b) A statistic (minimum nonzero eigenvalue) is computed for each distribution.}
\label{fig:all_pairs_torus}
\end{figure}

\section{Discussion}
\label{sec:discussion}
Complex problems often have data with interesting topological and geometric structure. Topological data analysis brings many quantitative and qualitative tools to understand the shape of the data. Persistent homology, the hallmark of TDA, efficiently computes large-scale shapes of data, and there are several vectorizations of persistent homology that allow it to shine in machine learning tasks. Persistent topological Laplacians, a spectral theory approach,  generalize both the graph Laplacian and aspects of persistent homology in algebraic topology to an operator that contributes additional geometric and structural information of the data. These have garnered attention from theoretical, algorithmic, and applied domains in recent years, although there remain significant challenges to their widespread adoption. Chief among these challenges is their scalability, as they are inherently computationally demanding. The structure has also been limited by the accessibility and flexibility of computational tools to handle distinct complexes, arbitrary dimensions, and the evolution of research into key algorithms. 

This work provides a framework for addressing these challenges independently of each other. In turn, we show improved scalability with practical examples and comparisons with existing implementations. We find that in many cases, computing the spectra, i.e., eigenvalues and eigenvectors,  is the primary obstruction, and resolutions may defy standard conventions. We propose a novel approach for using persistent homology to understand and compute persistent Laplacians, rather than using persistent Laplacians to inform persistent homology. Based on this framework, we can now also expand the types of analysis to use a larger and more rich family of persistent Laplacians. We can now develop algorithms independently, observe more nuanced eigenvalue behavior, and solve problems on larger scales. 

Although the present PETLS library deals with a few complexes, such as alpha,  directed flag, Dowker, and cellular sheaf, the developed framework allows for the further implementation of other topological spaces.

\section{Software availability}
The source code for the PETLS library can be accessed at the GitHub repository \url{https://github.com/bdjones13/PETLS} or as compiled Python binaries on the Python Packaging Index PyPI \url{https://pypi.org/project/petls/}. Documentation can be found at \url{https://BenJones-math.com/software/PETLS}.

\section*{Acknowledgment}	
This work was supported in part by NIH grant R35GM148196, NSF grant DMS-2052983, and MSU Research Foundation.

\appendix
\section{Complexes and filtrations}
\label{appendix:complexes}
We compute several variants of the persistent Laplacian, built on different complexes. 

\subsection{General interface}
All complexes ultimately proceed to a filtered boundary matrix structure. The minimum amount of information required to produce a complex for which we can compute persistent Laplacians is an ordered list of boundary matrices $\{d_1, d_2,\dots, d_N\}$ and a corresponding list of filtration values for the simplices in each dimension $\{F_0, F_1,\dots, F_N\}$, where each $F_i=\{f_i^0, f_i^1,\dots, f_i^{n_i}\}$. In practice the boundary matrices are a Python list of numpy arrays or a C++ vector of \texttt{Eigen::SparseMatrix<int>}. The filtrations are a Python list-of-lists of floats or a C++ vector-of-vectors of floats.

For example, consider a standard $2$-simplex as a simplicial complex with $0$-simplices $v_0, v_1, v_2$ with filtration values $0.0$, $1$-simplices $[v_0,v_1], [v_0,v_2],[v_1,v_2]$ with filtration values $0.1, 0.2,$ and $0.2$, respectively, and $2$-simplex $[v_0,v_1,v_2]$ with filtration value $1.4$. Then the boundary maps would be

\begin{align*}
d_1 &= \begin{blockarray}{ccccc}
        & & 0.1 & 0.2 & 0.2 \\
        & & [v_0,v_1] & [v_0,v_2] & [v_1,v_2] \\
        \begin{block}{cc(ccc)}
          0.0 & v_0 & -1 & -1 & 0 \\
          0.0 & v_1 & 1 & 0 & -1 \\
          0.0 & v_2 & 0 & 1 & 1 \\
        \end{block}
    \end{blockarray}\\
d_2 &= \begin{blockarray}{ccc}
        & & 1.4 \\
        & & [v_0,v_1,v_2]\\
        \begin{block}{cc(c)}
          0.1 & [v_0,v_1] & 1 \\
          0.2 & [v_0,v_2] &-1 \\
          0.2 & [v_1,v_2] &1\\
      \end{block}
    \end{blockarray}.
\end{align*}

The library would then be given the matrices as
\[\texttt{boundaries}=\left\{\begin{pmatrix}
    -1 & -1 & 0\\
    1 & 0 & -1\\
    0 & 1 & 1
\end{pmatrix}, \begin{pmatrix}
    1\\ -1\\ 1
\end{pmatrix}\right\}\]

and the filtrations as 

\[\texttt{filtrations}=\{[0.0,0.0,0.0], [0.1, 0.2, 0.2],[1.4]\}.\]

A key purpose of this representation is that it is immediately compatible with complexes that are non-simplicial, enabling computation of persistent Laplacians for path complexes \cite{grigoryan_path_2020,wang_persistent_2023}, hyperdigraph complexes \cite{chen_persistent_2023}, and the directed flag complex (an ordered simplicial complex) \cite{jones_persistent_2025,lutgehetmann_computing_2020}.

The Python package provides the additional functionality to produce this filtered boundary matrix structure from any Gudhi Simplex Tree \cite{gudhi:urm}, a commonly used and efficient data structure that can represent any simplicial complex filtration.

\subsection{Vietoris-Rips}
The Vietoris-Rips (VR) filtration is the most common distance-based filtration for producing simplicial complexes of point cloud data. Given a finite metric space $(X,d)$, usually a set of points in $\mathbb{R}^d$ with the Euclidean distance, the Vietoris-Rips complex with parameter $\epsilon$ is $VR_\epsilon(X)=\{\emptyset\neq S\subset X\mid \mathrm{diam } S\leq \epsilon\}$. For $\epsilon\leq \tau$, this produces a filtration of complexes $VR_\epsilon(X)\subset VR_{\tau}(X)$. In Figure \ref{fig:rips} we show the VR filtration for a point cloud $X=\{a,b,c\}$ consisting of points at the vertices of a $3-4-5$ triangle. When $\epsilon < 3$, the only subsets of $X$ with diameter less than $\epsilon$ are the points themselves with diameter $0$. Then at $\epsilon=3$ the $1$-simplex $[a,b]$ is added, since the diameter of $\mathrm{diam}\{a,b\}= 3$, and at $\epsilon=5$ the $1$-simplex $[b,c]$ and $2$-simplex $[a,b,c]$ are added. For $\epsilon>5$ the complex is unchanged. In practical applications, a maximum value of $\epsilon$ is usually set, because the number of $k$-simplices in an non-thresholded VR complex on $n$ points grows at $\Omega(n^k)$ \cite{dey2022computational}.

\begin{figure}[htpb!]
    \centering
    \includegraphics[width=5.069in]{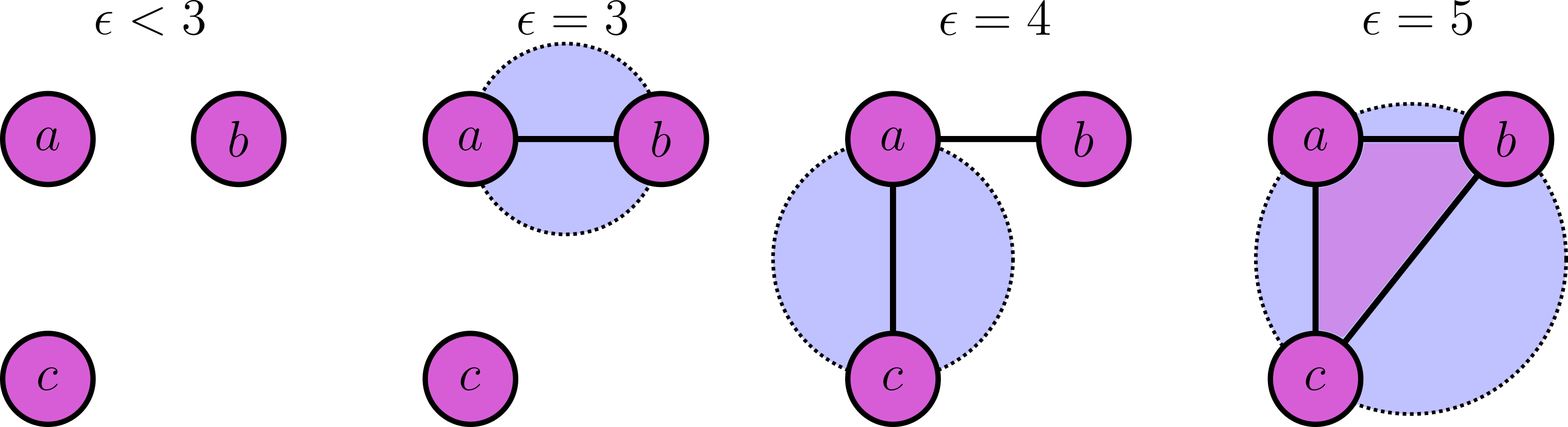}
    \caption{A Vietoris-Rips filtration for a point cloud consisting of points at the vertices of a $3-4-5$ triangle.}
    \label{fig:rips}
\end{figure}

The VR complex depends only on the pairwise distances between points, and therefore can be computed from a distance matrix with no additional information. This also allows for points embedded in an arbitrary ambient dimension. We modify Ulrich Bauer's Ripser \cite{bauer2021ripser} to use real coefficients and produce an explicit filtered boundary matrix. The complex can be built via several methods, such as \verb|rips = petls.Rips(points = my_points)|.

\subsection{Alpha complex}

An Alpha complex is a commonly used simplicial complex for modeling point cloud data in $\mathbb{R}^2$ and $\mathbb{R}^3$. We give a brief recap of the more detailed description in \cite{dey2022computational}. The Alpha complex is a certain subcomplex of the Delaunay complex. The Delaunay complex $\operatorname{Del} P$ of a finite point set $P$ in $\mathbb{R}^d$ in general position is a geometric simplicial complex where for each simplex $\sigma$ there is an open $d$-ball that does not contain any point in $P$ in its interior, but the boundary of the ball contains the vertices of $\sigma$, and where $|\operatorname{Del} P|$ is the convex hull of $P$. There is one such complex. The Alpha complex for $\alpha\geq 0$ is the complex formed by simplices $\sigma$ in $\operatorname{Del} P$ where $\sigma$ has a circumscribing ball of radius at most $\alpha$. The Delaunay complex has a duality relationship with the Voronoi diagram. We use the Alpha complex built by GUDHI \cite{gudhi:urm}, which uses CGAL's Delaunay triangulation \cite{cgal:hs-chdt3-24b}.  Gudhi (and therefore PETLS) uses a convention of giving each simplex a filtration value of the \textit{square} of the circumradius, $\alpha^2$. 

For an example of a $2$D Alpha complex, consider points at the vertices of a $2-4-5$ obtuse triangle in Figure \ref{fig:alpha}. The edge $[b,c]$ is added at filtration value $\alpha^2=1$, the edge $[a,b]$ at $\alpha^2=4$, and both the edge $[a,c]$ and triangle $[a,b,c]$ are added at $\alpha^2\approx 6.93$. Note that the edge $[a,c]$ with side length $5$ is not added at $\alpha^2=\left(\frac{5}{2}\right)^2=6.25$, since the circle with diameter $[a,c]$ includes the point $b$ in the interior, violating the restriction on the open ball.

\begin{figure}[htpb!]
    \centering
    \includegraphics[width=5.189in]{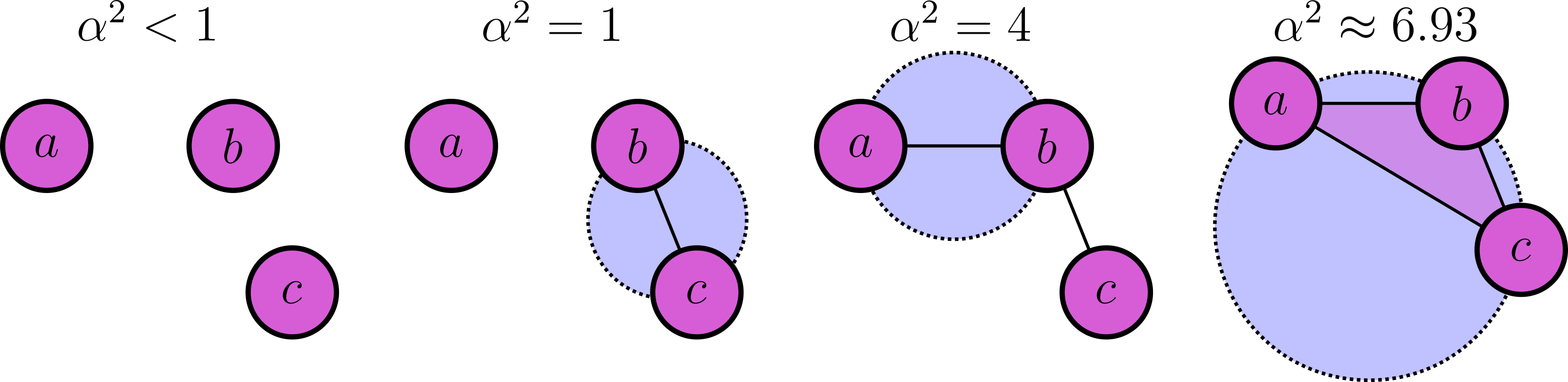}
    \caption{Alpha filtration for a point cloud consisting of points at the vertices of a $2-4-5$ obtuse triangle.}
    \label{fig:alpha}
\end{figure}

\subsection{Directed flag (Clique) complex}

A directed flag (clique) complex is a directed simplicial complex used to model higher-order interactions of directed graphs. A $k$-simplex $\sigma=[v_1,v_2,\dots,v_k]$ is in the directed flag complex $dFl(G)$ of a directed graph $G=(V,E)$ if $[v_i,v_j]$ is an edge for all $1\leq i < j\leq k$. Each $k$-simplex is a complete $(k+1)$-subgraph with a single source vertex ($v_1$) and a single sink vertex ($v_k$). Figure \ref{fig:dflag} shows the directed graph motifs that result in simplices for dimensions $0,1,2,$ and $3$. Vertices and edges are the $0$- and $1$-simplices, while the $2$-simplex is a triangle with the specific orientation of the edges. The other distinct orientation (replacing $[v_1,v_3]$ with $[v_3,v_1]$) results in a homological cycle which is not a boundary, and therefore the directed flag complex is sensitive to orientation and asymmetry. A filtered directed flag complex can have any filtration that respects inclusion. We use a modification \cite{jones_persistent_2025} 
of the flagser software  \cite{lutgehetmann_computing_2020}, which is based on Ripser \cite{bauer2021ripser} and computes the persistent homology of directed flag complexes.

\begin{figure}[htpb!]
    \centering
    \includegraphics[width=4.0in]{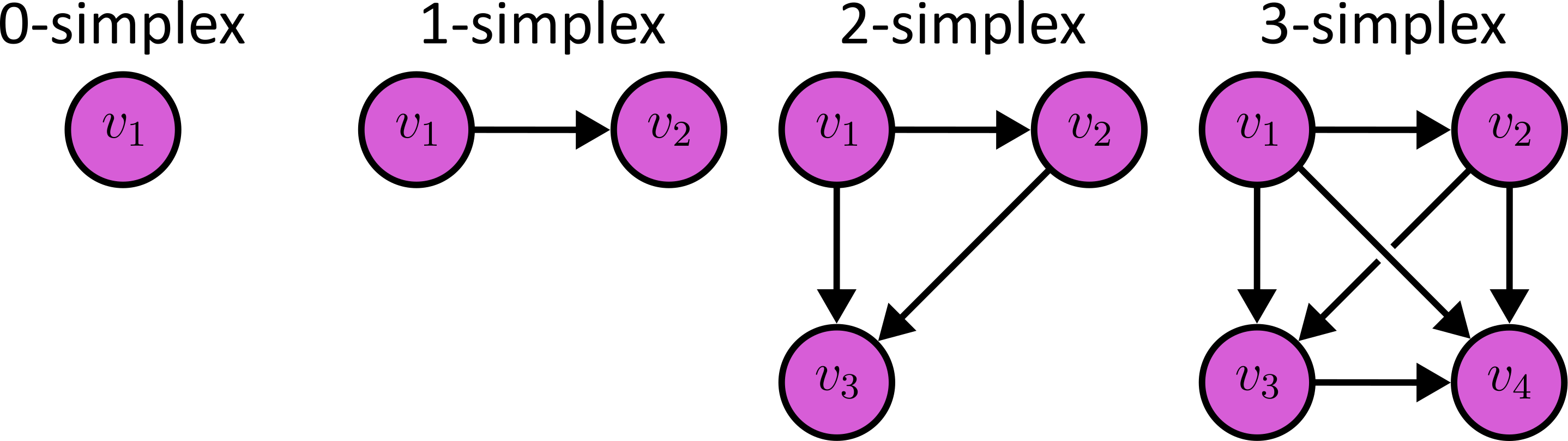}
    \caption{Graph templates for simplices in a directed graph.}
    \label{fig:dflag}
\end{figure}

\subsection{Dowker complex}
The Dowker complex \cite{dowker1952homology} and its persistent version \cite{chowdhury2018functorial} create a simplicial complex meant to encode the structure of a relation, which is possibly asymmetric, between two sets. The (persistent) homology of this complex may reveal structure of the relation. 

Using the definitions and notation of \cite{chowdhury2018functorial}, suppose $X$ and $Y$ are two totally ordered sets and $R\subset X\times Y$ is a nonempty relation. There are \textit{two} simplicial complexes $E_R$ and $F_R$, which have isomorphic homology, $H_k(E_R)\cong H_k(F_R)$, and the geometric realizations of the complexes are homotopy equivalent. Chowdhury and M{\'e}moli \cite{chowdhury2018functorial} gave a functorial version of the isomorphism and introduced Dowker persistence, which can detect asymmetry in networks. Because of these equivalences, many authors refer to ``the" Dowker complex. The (persistent) Laplacian is not invariant up to homotopy, and therefore the Dowker complexes can have distinct persistent Laplacians and eigenvalues. Possible relationships between the eigenvalues of the (persistent) Laplacians of the two Dowker complexes may be of interest. 

A simplex $\sigma\subset X$ is in $E_R$ if there exists $y\in Y$ such that $(x,y)\in R$ for every $x\in \sigma$, and analogously $\tau\subset Y$ is in $F_R$ if there exists $x\in X$ such that $(x,y)\in R$ for every $y\in\tau$. For example, see Figure \ref{fig:dowker}(a), where the sets $X=\{a,b,c\}$ and $Y=\{i,j,k\}$ have a relation $R\subset X\times Y$, which can be expressed as a matrix:

\[
R = \begin{blockarray}{cccc}
            & i & j & k \\
            \begin{block}{c(ccc)}
                {a} & 1 & 1 & 1\\
                {b} & 1 & 1 & 0 \\
                {c} & 0 & 0 & 1 \\
            \end{block}
        \end{blockarray}
\]

Then we can find that $E_R=\{\emptyset, [a],[b],[c],[a,b],[a,c]\}$ (notably excluding $[b,c]$ and $[a,b,c]$) and $F_R=\{\emptyset,[i],[j],[k],[i,j],[i,k],[j,k],[i,j,k]\}$. The geometric realizations of these simplicial complexes are in Figure \ref{fig:dowker}(b). Both are contractible, and therefore homotopy equivalent with trivial homology.

\begin{figure}[htpb]
    \centering
    \begin{subfigure}{1.3in}
       \centering
        \includegraphics[width=1.3in]{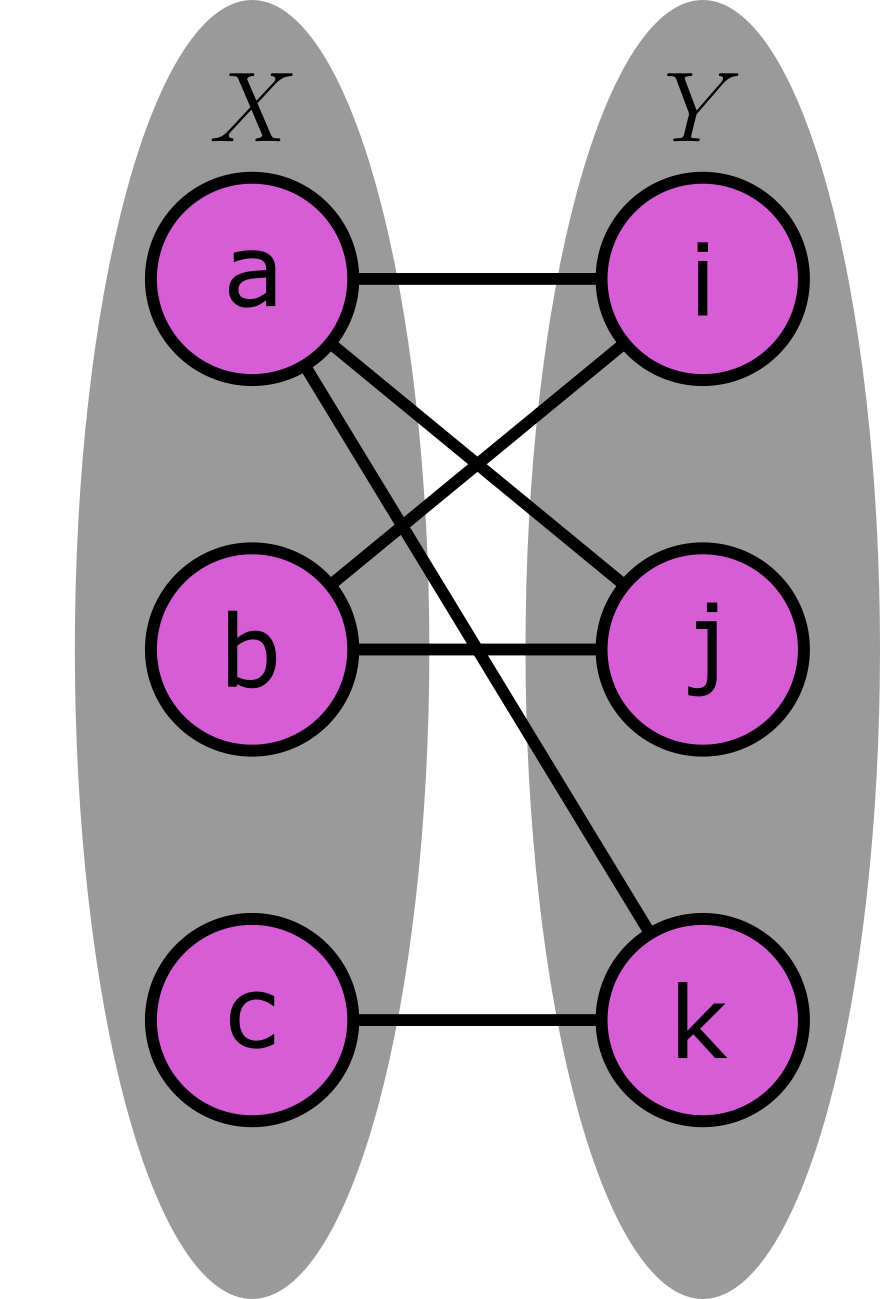}
        \caption{}
    \end{subfigure}
    \begin{subfigure}{2.2in}
        \centering
        \includegraphics[width=2.2in]{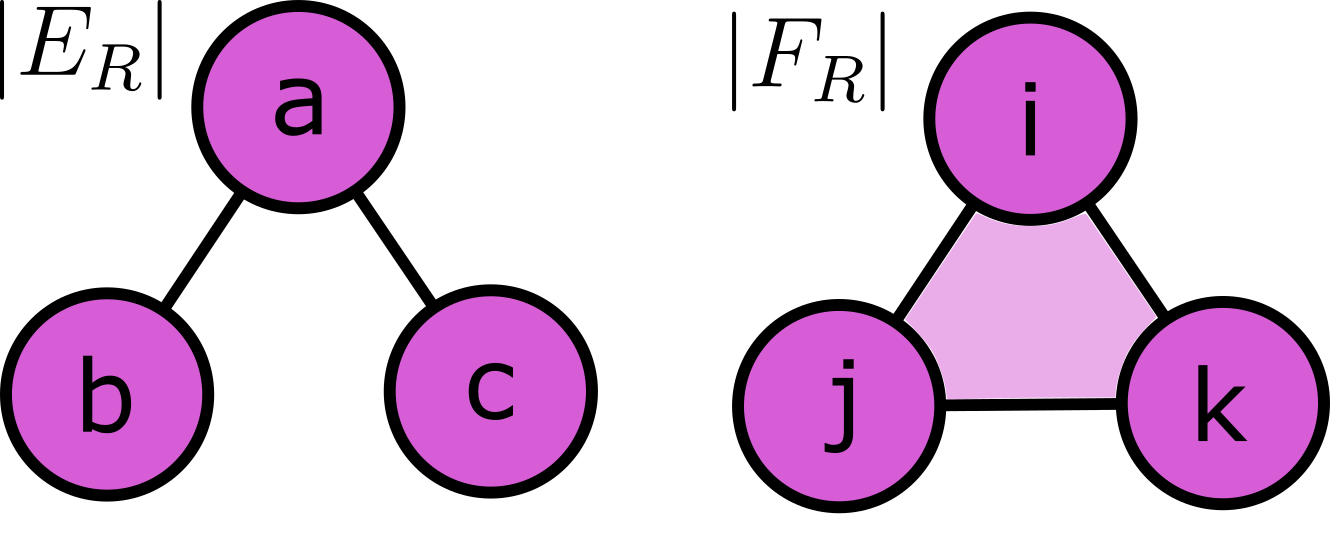}
        \caption{}
    \end{subfigure}
\caption{(a) Visual representation of a relation $R$ between sets $X$ and $Y$. (b) Geometric realizations of the two Dowker complexes corresponding to $R$.}
\label{fig:dowker}
\end{figure}

For a weighted directed graph $(X,\omega)$, one can define a filtration of relations $R_{\delta, X}=\{(x,x'):\omega(x,x')\leq \delta\}\subset X\times X$. Then the Dowker sink filtration \cite{chowdhury2018functorial} is 
\[
\mathfrak{D}_{\delta, X}^{\mathrm{si}}=\{\sigma=[x_0,\dots,x_n]:\text{there exists } x'\in X \text{ such that }(x_i,x')\in R_{\delta,X}\text{ for each }x_i\}.
\]

That is, there exists $x'\in X$ such that from each vertex $x_i$ in the simplex, there is an edge from $x_i$ to $x'$ meeting the filtration threshold $\delta$. An analogous Dowker source filtration exists and the persistence diagrams of the Dowker source and sink filtrations are equal.

Filtered Dowker complexes are implemented in pyDowker \cite{hellmer2024density}, which produces a compatible simplex tree that can be immediately used by PETLS to construct a persistent Laplacian.

\subsection{Cellular sheaf}\label{appendix:sheaf} 
A cellular sheaf is an additional structure associated with a simplicial (or cell) complex. A comprehensive discussion is in Justin Curry's dissertation \cite{curry2013sheaves}. An extension to combinatorial Laplacians is in \cite{hansen2019toward} and an extension to persistence is in \cite{wei_sheaf_2025}. We will give an abbreviated version of the relevant pieces. The theory of cellular sheaves is defined for more general cell complexes than we require for applications involving point cloud data, so we will restrict our discussion to simplicial complexes. The distinctions are explained in \cite{curry2013sheaves}.

A cellular sheaf $\mathscr{S}$ on a simplicial complex $X$ consists of vector spaces $\mathscr{S}(\sigma)$ for each $\sigma\in X$ called stalks and for each face relationship $\sigma\leq\tau$ (where $\sigma$ is a face of $\tau$) a linear map $\mathscr{S}_{\sigma\leq\tau}:\mathscr{S}(\sigma)\to\mathscr{S}(\tau)$ called the restriction map. The restriction map must satisfy a composition rule: if $\rho\leq\sigma\leq\tau$ then $\mathscr{S}_{\rho\leq\tau}=\mathscr{S}_{\sigma\leq\tau}\circ\mathscr{S}_{\rho\leq\sigma}$. It must also satisfy $\mathscr{S}_{\sigma\leq\sigma}=\mathrm{id}$. More succinctly as said in \cite{hansen2019toward}, if $P_X=(X,\leq)$ is the poset where $\leq$ is the face relation,  then a cellular sheaf is a functor $\mathscr{S}:P_X\to \mathbf{Vect}_{\mathbb{K}}$, for a field $\mathbb{K}$ (which we will assume is $\mathbb{R}$).

To define the relevant cohomology theory, we use the cochain spaces as $C^n(X;\mathscr{S})=\bigoplus_{\dim(\sigma)=n}\mathscr{S}(\sigma)$. The coboundary map of cellular sheaves traditionally relies on the notion of a signed incidence relation to encode the orientation of simplices \cite{curry2013sheaves}. In our case, the vertices in a finite simplicial complex will be ordered, so the ordering of vertices in a simplex determine the orientation of the simplex, so we can use a signed incidence relation to that is the same as the standard notion of assigning signs \cite{wei_sheaf_2025}. That is, for $\tau = [v_0,v_1,\dots, v_n]$ and $\sigma=[v_0,v_1,\dots,\hat{v_i},\dots, v_n]$, the sign is $[\sigma:\tau]=(-1)^{i}$, and the coboundary map is $d^n:C^n(X;\mathscr{S})\to C^{n+1}(X;\mathscr{S})$ given by $d^n\mid_{\mathscr{S}(\sigma)}=\sum_{\sigma\leq\tau}[\sigma:\tau]\mathscr{S}_{\sigma\leq\tau}$. Then $d^n\circ d^{n-1}=0$ and $(C^\bullet(X;\mathscr{S}), d^\bullet)$ is a cochain complex \cite{curry2013sheaves} with cohomology $H^n(X;\mathscr{S})=\ker d^{n}/\Ima d^{n-1}$. 

The persistent Laplacian is normally constructed for a chain complex, rather than a cochain complex. However, in our setting with stalks that are finite dimensional inner product spaces (e.g. copies of $\mathbb{R}^d$), the stalks are self-dual, and we can define the persistent Laplacian as usual on a dualized diagram \cite{wei_sheaf_2025}.

Suppose we have the simplicial complex on vertices $v_0, v_1, v_2$ in Figure \ref{fig:cellular_sheaf}(a). In Figure \ref{fig:cellular_sheaf}(b) we show the corresponding stalks for a general cellular sheaf $\mathscr{S}$. In Figure \ref{fig:cellular_sheaf}(c) we add the restriction maps in blue. In Figure \ref{fig:cellular_sheaf}(d) we provide explicit spaces and maps for the stalks and restrictions. Note the consistency demonstrated in the two different routes of composing to $\mathscr{S}_{[0]\leq[012]}$:

\begin{gather*}
    \mathscr{S}_{[01]\leq[012]}\circ\mathscr{S}_{[0]\leq[01]} = \mathscr{S}_{[0]\leq [012]}= \mathscr{S}_{[02]\leq[012]}\circ\mathscr{S}_{[0]\leq[02]}\\
    \begin{bmatrix}
        5 & 1
    \end{bmatrix}\begin{bmatrix}
        1 \\ 1
    \end{bmatrix}=6 = \begin{bmatrix}
        3
    \end{bmatrix}\begin{bmatrix}
        2
    \end{bmatrix}
\end{gather*}

\begin{figure}[htpb]
    \centering
    \begin{subfigure}{2.336in}
       \centering
        \includegraphics[width=2.336in]{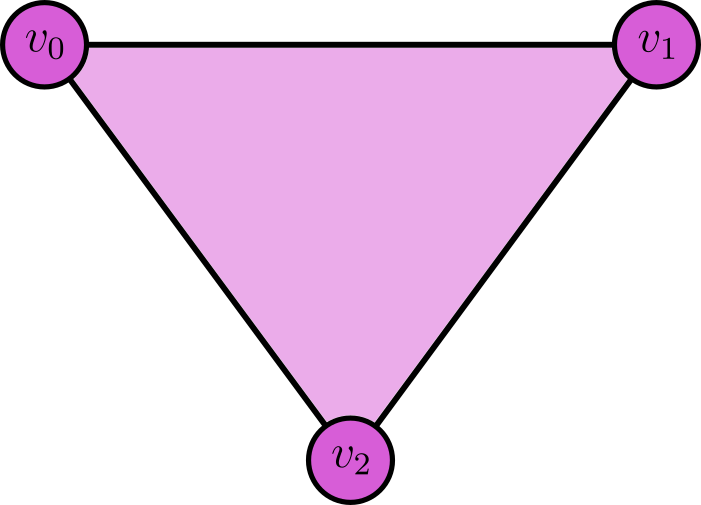}
        \caption{}
    \end{subfigure}
    \hfill
    \begin{subfigure}{3.513in}
        \centering
        \includegraphics[width=3.513in]{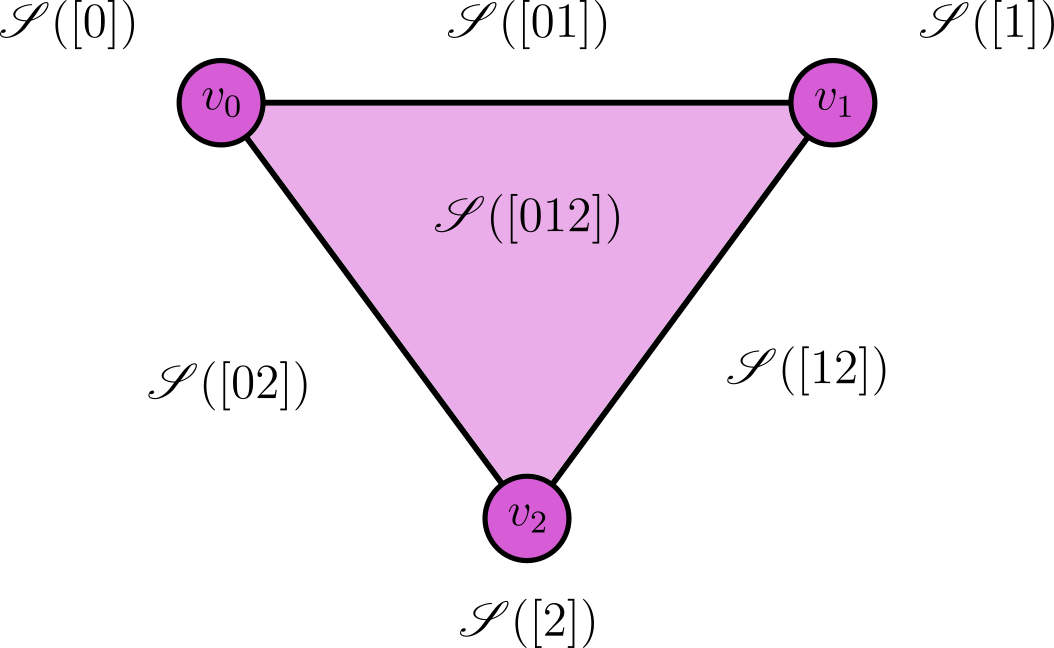}
        \caption{}
    \end{subfigure}
    \begin{subfigure}{3.287in}
        \centering
        \includegraphics[width=3.287in]{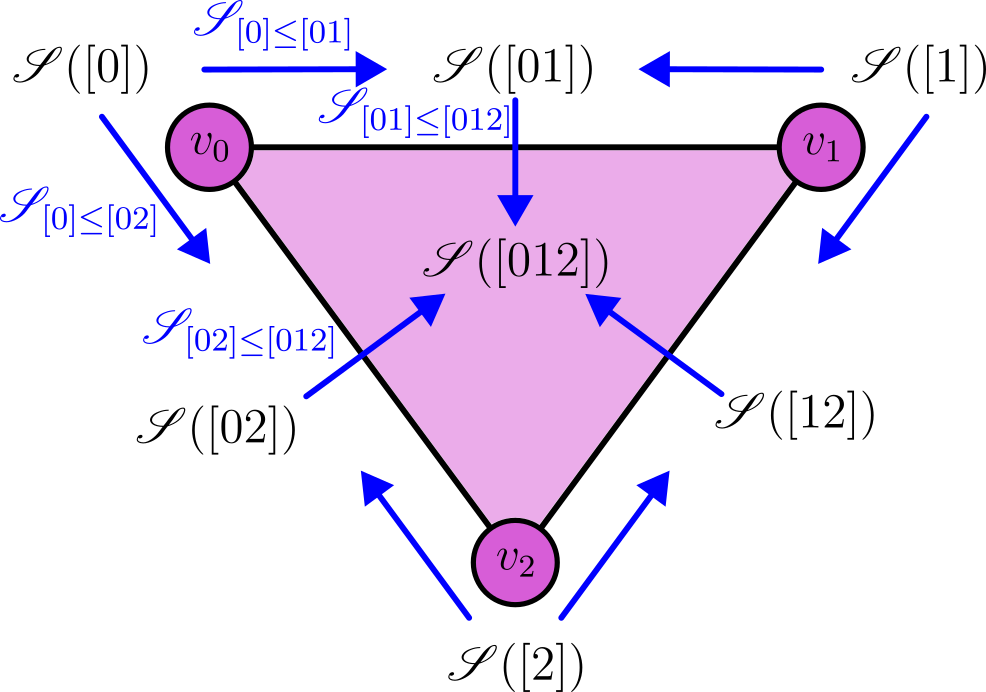}
        \caption{}
    \end{subfigure}
    \hfill
    \begin{subfigure}{2.888in}
        \centering
        \includegraphics[width=2.888in]{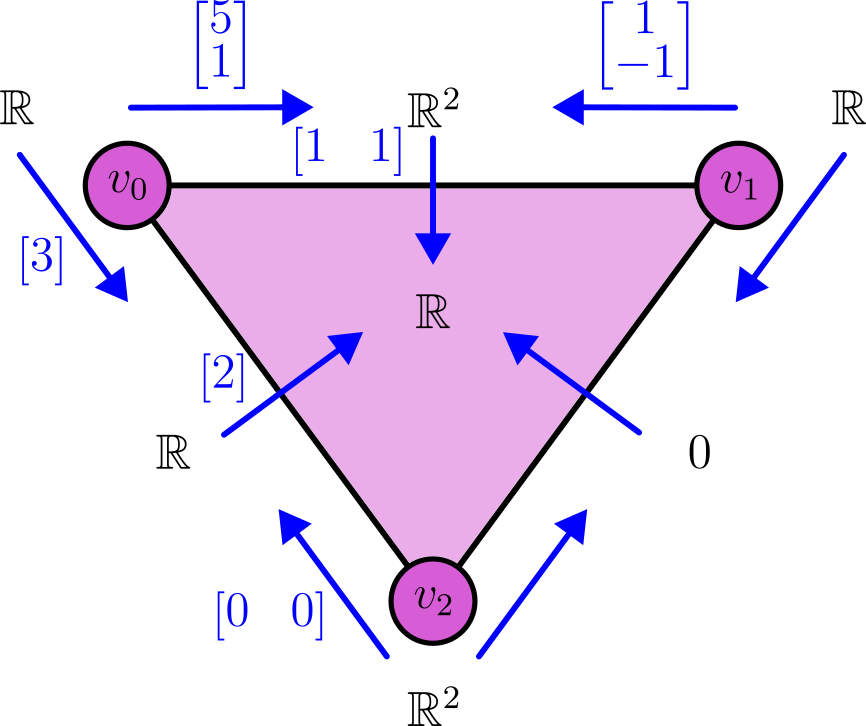}
        \caption{}
    \end{subfigure}
\caption{A cellular sheaf on a simplicial complex. (a) the simplicial complex with three vertices, three edges, and one $2$-simplex. (b) The stalks $\mathscr{S}(\sigma)$ for each simplex. Note that we have written $\sigma=[ijk]$ rather than $\sigma=[v_i,v_j,v_k]$ for clarity. (c) The restriction maps are added in blue. (d) explicit stalks and restriction maps are given.}
\label{fig:cellular_sheaf}
\end{figure}

 \bibliographystyle{abbrv}
\bibliography{main}

\end{document}